 \documentclass[12pt]{amsart}

\usepackage{amsmath}
\usepackage{amssymb} 
\usepackage{array}
\usepackage{enumitem}
\usepackage{hyperref}
\usepackage[all]{xy}
\usepackage{verbatim} 
\usepackage{lscape} 
\usepackage{color}
\definecolor{battleshipgrey}{rgb}{0.52, 0.52, 0.51} 
\hypersetup{
       colorlinks=true,       
    linkcolor=battleshipgrey,     
    citecolor=battleshipgrey,        
    filecolor=magenta,      
    urlcolor=battleshipgrey          
}
\usepackage{pstricks-add}

\theoremstyle{plain}
\newtheorem{theorem}{Theorem}[section]
\newtheorem{lemma}[theorem]{Lemma}
\newtheorem{proposition}[theorem]{Proposition}
\newtheorem{corollary}[theorem]{Corollary}

\newtheorem{definition}[theorem]{Definition}
\theoremstyle{remark}
\newtheorem{remark}{Remark}[section]
\newtheorem{example}{Example}[section]
\newtheorem*{notation}{Notation}

\numberwithin{equation}{section}

\newcommand{\bA}{\mathbb{A}}

\newcommand{\bB}{\mathbb{B}}
\newcommand{\bC}{\mathbb{C}}

\newcommand{\K}{\mathbb{K}}

\newcommand{\R}{\mathbb{R}}
\newcommand{\Z}{\mathbb{Z}}

\newcommand{\N}{\mathbb{N}}

\newcommand{\F}{\mathbb{F}}

\newcommand{\bH}{\mathbb{H}}

\newcommand{\cX}{\mathcal{X}}

\newcommand{\cP}{{\mathcal P}}

\newcommand{\s}{\mathfrak{S}}

\newcommand{\sss}{\mathbf{s}}

\newcommand{\ttt}{\mathbf{t}}

\newcommand{\sfn}{\mathsf{n}} 

\newcommand{\sfone}{\mathsf{1}}
\newcommand{\sftwo}{\mathsf{2}}

\newcommand{\Gl}{\mathrm{Gl}}

\newcommand{\Aut}{\mathrm{Aut}}
\newcommand{\Inv}{\mathrm{Inv}}

\newcommand{\opp}{\mathrm{opp}}

\newcommand{\im}{\mathrm{im}}

\newcommand{\id}{\mathrm{id}}

\newcommand{\sign}{\mathrm{sign}}

\newcommand{\Cl}{{\mathrm{Cl}}}

\newcommand{\wtimes}{\, \widehat \times}

\newcommand{\msk}{\medskip}
\newcommand{\ssk}{\smallskip}
\newcommand{\nin}{\noindent}

\newcommand{\ul}{\underline}
\newcommand{\ol}{\overline}

\begin{document}

\title[ ]{
Graded sets, graded groups, \\
and Clifford algebras}

\author{Wolfgang Bertram}

\address{Institut \'{E}lie Cartan de Lorraine \\
Universit\'{e} de Lorraine at Nancy, CNRS, INRIA \\
Boulevard des Aiguillettes, B.P. 239 \\
F-54506 Vand\oe{}uvre-l\`{e}s-Nancy, France\\
\url{http://wolfgang.bertram.perso.math.cnrs.fr}
}

\email{\url{wolfgang.bertram@univ-lorraine.fr}}



\begin{abstract}  
We define a general notion of
{\em centrally $\Gamma$-graded sets and groups} and of their
{\em graded products}, and prove some basic results about
the corresponding categories: most importantly, they form
{\em braided monoidal categories}.
Here, $\Gamma$ is an arbitrary (generalized) ring.
The case $\Gamma = \Z / 2 \Z$ is studied in detail: it is
related to Clifford algebras and their discrete Clifford groups
(also called Salingaros Vee groups).
\end{abstract}

\maketitle

\begin{center}
{\sl
Dedicated to the memory of my collegue 
Lionel B\'erard-Bergery
(1945 -- 2019)}
\end{center}

\setcounter{tocdepth}{1}

\section*{Introduction}

Let $(\Gamma,+)$ be an abelian group (our main emphasis will be on the case
$\Gamma = \Z/2 \Z$). Generally,
{\em $\Gamma$-graded} structures play an important r\^ole in  mathematics. 
In this work we study a class of $\Gamma$-gradings, which, when combined with
what we call a {\em central $\Gamma$-action}, has very interesting and pleasant properties.

\subsection{Additively and multiplicatively graded sets}
A first point in the present work is to distinguish between {\em additively}, and {\em multiplicatively} graded
structures:

\begin{definition}\label{def:add-graded}
Let $X$ and $M$ be  sets.  
\begin{enumerate}
\item
An {\em additive $\Gamma$-grading on $M$} is a  {\em disjoint union structure},
written
$$
M = \sqcup_{\gamma \in \Gamma} M_\gamma, \mbox{ or }
M = \coprod_{\gamma \in \Gamma} M_\gamma,
$$
i.e., $M$ is the  {\em coproduct} of a family of sets
$(M_\gamma)_{\gamma \in \Gamma}$.
We allow some of the $M_\gamma$ to be empty.
Thus, an additive grading is equivalent to a 
{\em degree-map},
$$
d : M \to \Gamma , \quad \mbox{ such that } \quad
M_\gamma = d^{-1} (\gamma) .
$$
\item
A {\em multiplicative $\Gamma$-grading on} $X$ is  a
{\em direct product structure} of the form 
$$
X = \times_{\gamma \in \Gamma} X_\gamma, \mbox{ or }
X = \prod_{\gamma \in \Gamma} X_\gamma .
$$
\end{enumerate}
\end{definition}

If $M$ is additively graded, then the space $X = K^M$ of functions from $M$ to a set $K$
is multiplicatively graded
(a function $f:M \to K$ corresponds to the family $(f\vert_{M_\gamma})_{\gamma \in \Gamma}$); 
in particular, taking $K = \{ 0,1\}$, we see that
the power set $\cP(M)$ of $M$ is multiplicatively graded: 
multiplicative gradings often arise as ``exponentials'' of additive ones.  This furnishes important links 
between both types of gradings. 

\subsection{Graded groups, and Clifford groups}
In this work, we focus on {\em additively} graded sets, and in particular on 
{\em (additively) graded groups or monoids}:

\begin{definition}\label{def:graded-monoid} 
A {\em $\Gamma$-graded monoid} is a monoid $G$ which is additively $\Gamma$-graded,
such that the grading map $d:G \to (\Gamma,+)$ is a morphism of monoids.
It is called {\em centrally graded} given
a morphism
$Z :(\Gamma,+) \to G$, $x \mapsto Z(x) = :Z^x$ such that:
\begin{enumerate}
\item
 $Z$ takes
values in the center
of $G$: for all $x \in \Gamma$ and $g\in G$ : $Z^x g = g Z^x$,
\item
  $d \circ Z = 0$, i.e., 
$\forall x \in \Gamma$ : $d(Z^x)= 0$.
\end{enumerate}
\end{definition}

When $\Gamma = \Z/2 \Z = \{ \ol 0,\ol 1\}$, then the element $Z=Z(\ol 1)$ is called the
{\em grading element}, and the power $Z^x$ is the same as $Z(x)$. 
A prototype of graded group $G$ then is
what in certain texts
(e.g., \cite{L}) is called the {\em (discrete) Clifford group},
and elsewhere  {\em Salingaros Vee-groups}, (cf.\
\cite{S, S82, A}):  
the canonical basis elements
$e_1,\ldots,e_n$ of the Clifford algebra $\Cl_{p,q}(\K)$ (see Appendix \ref{app:Clifford} for notation) 
 generate
a group $G=Q_{p,q}$ of invertible elements in the Clifford algebra.
This group is $\Gamma$-graded in the sense defined above:
the grading morphism $d$ is the unique morphism $d:Q_{p,q} \to \Z/2 \Z$ 
such that $d(e_i) = \ol 1$ (odd) for all $i$,
and the grading element is $Z = - 1$ (since $\langle e_i,e_j \rangle = 0$ for $i \not= j$, we have 
$e_i e_j = - e_j e_i$, so the group $G$ contains the element
$ -1$, which clearly is central in $G$).
The underlying philosophy is to replace the central element
$-1$ from the Clifford case by an ``abstract central element''.
Thus the approach presented here will give an abstract, group-theoretic  presentation
of the discrete Clifford group. 
Such a group-theoretic approach has already been advocated by the above
quoted authors, but here we will go much further by putting it into a general
categorical approach, as follows. 

\subsection{The braided monoidal category of graded sets}
A key feature in the theory of Clifford algebras is the relation 
(see, e.g., \cite{BtD}, Prop.\ I.6.6)
 \begin{equation}\label{eqn:Clifford-tensorproduct}
 \Cl(U \oplus W) = \Cl(U) \, \widehat \otimes \,  \Cl(W)
 \end{equation}
 expressing the Clifford algebra of a quadratic space $U \oplus W$ by the
 {\em graded tensor product} of the corresponding Clifford algebras. When $U,W$
 carry the zero quadratic form, we get the corresponding rules for
 {\em Grassmann algebras}. 
 Our approach is designed to furnish an analog of this for general (additvely)
 {\em centrally $\Gamma$-graded sets}:

\begin{definition}\label{def:centrally-graded-set}
A $\Gamma$-graded set, with grading map $d:M\to \Gamma$, is called {\em centrally
$\Gamma$-graded} if it is equipped with a (left) $\Gamma$-action
$$
z : \Gamma \times M \to M, \quad (x,m) \mapsto z(x,m) = x.m
$$
preserving the grading, i.e.,
such that $\forall x \in \Gamma, \forall m \in M$ : $d(x.m)=d(m)$.
\end{definition}

A centrally graded monoid $G$ is a centrally graded set: just let
$z(x,m) = Z^x m$.
First of all, we define the graded product of sets:

\begin{theorem}\label{th:monoidal-category}
Assume $d_i : M_i \to \Gamma$, $i=1,2$, are two centrally $\Gamma$-graded sets.
Let $M$ be the quotient of $M_1 \times M_2$ for the equivalence relation defined by
$\forall x \in \Gamma$ : $(x.m_1,m_2) \sim (m_1,x.m_2)$. 
Then $M$, together with grading map and $\Gamma$-action
$$
d([m_1,m_2]) = d(m_1) + d(m_2), \quad x . [m_1,m_2] = [x.m_1,m_2]=[m_1,x.m_2],
$$
is again a centrally $\Gamma$-graded set, which we denote by
$M = M_1  \times_z M_2$.  The usual
set-theoretic identification
$(M_1 \times M_2) \times M_3 = M_1 \times (M_2 \times M_3)$ induces
a natural isomorphism of centrally $\Gamma$-graded sets
$$
(M_1 \times_z M_2) \times_z M_3 = M_1 \times_z (M_2 \times_z M_3).
$$
The set $E=\Gamma$ with $d(m)=0$, $x.m = x+m$, is a ``unit'' :
$E \widehat \times_z M = M = M  \times_z E$.
Summing up, centrally $\Gamma$-graded sets together with 
$\times_z$ form a {\em monoidal category}.
\end{theorem}

Now, the crucial point about graded tensor products is that they are associative, but {\em not
commutative}: this is encoded by  the
{\em braiding map} of graded tensor products, defined for homogeneous elements $v,w$
of degree $\vert v \vert, \vert w \vert$ by
\begin{equation}\label{eqn:signrule}
\beta : 
V \widehat\otimes W \to W \widehat\otimes V , \quad
v \otimes w \mapsto (-1)^{\vert v \vert \cdot \vert w \vert} w \otimes v.
\end{equation}
Likewise, in our setting, braidings can be chosen as additional structure: first of all,
there is the ``usual'', or ``standard'' braiding, $(x_1,x_2)\mapsto (x_2,x_1)$,
which leads us back to usual set-theory.
But now assume that $\Gamma$ is also a {\em generalized ring}, meaning that
it carries a bi-additive ``product'' map $\Gamma^2 \to \Gamma$, $(a,b) \mapsto ab$.

\begin{definition} \label{def:braiding}
Let $M_i$, $i=1,2$, be two centrally $\Gamma$-graded sets. Then the map
$$
\beta = \beta_{M_1,M_2} : M_1 \times_z M_2 \to M_2 \times_z M_1,\quad
[x_1,x_2] \mapsto z\bigl(  d(x_2) \cdot d(x_1) ,  [x_2,x_1]  \bigr)
$$
is well-defined, called the {\em braiding} given by the generalized ring $\Gamma$.
\end{definition}

\begin{theorem}\label{th:braiding}
For any generalized ring, the braiding maps are isomorphisms of centrally graded sets and
define a structure of 
{\em (strict) braided monoidal category} of centrally $\Gamma$-graded sets.
If the generalized ring product is {\em skew}, meaning that
$ab+ba = 0$ for all $a,b \in \Gamma$, then this structure is {\em symmetric},
i.e., we always have
$$
\beta_{M_1,M_2}^{-1} = \beta_{M_2,M_1}.
$$
\end{theorem}

In the symmetric case, there is a simple formula, in terms of {\em permutation inversions},
 for the action of the symmetric group on
iterated products (Theorem \ref{th:braided2}). 
Next, braiding maps are used to define {\em graded products of graded groups}:

\begin{theorem}\label{th:graded-group}
We fix a generalized ring structure on $\Gamma$.
\begin{enumerate}
\item
For every centrally $\Gamma$-graded group $G$ there is a {\em braided dual group ($b$-dual)} 
given by $G^\vee = G$, with same grading and central action as $G$, and product
$$
x \bullet y = Z^{d(y) d(x)} \, xy.
$$
\item
There is a group structure,  denoted
 by $G_1 \widehat \times_Z G_2$, on the product
$G_1 \times_z G_2$ of two centrally $\Gamma$-graded groups $G_i$, $i=1,2$, given by the group law
$$
[g_1,g_2]\cdot [h_1,h_2] = Z^{d(h_1) d(g_2)} [g_1h_1,g_2 h_2] .
$$
\item
Iteration of the preceding construction is associative:  the natural map

$(G_1 \widehat \times_Z G_2)\widehat \times_Z G_3 \cong 
 G_1 \widehat \times_Z (G_2\widehat \times_Z G_3)  $
is a group isomorphism.
\item
When $\Gamma$ is skew, then the braiding map $\beta_{G_1,G_2}$ is an isomorphism from
$G_1 \widehat \times_Z G_2$ onto 
$G_2 \widehat \times_Z G_1$.
\end{enumerate}
\end{theorem} 

Dually to the result mentioned above, there is also a simple formula for the product of $n$ elements
in $G_1 \widehat \times_Z G_2$ (Remark \ref{rk:assoc-2}) and in $G^\vee$ (Remark \ref{rk:assoc-0}).

\subsection{The discrete Clifford category}
Turning to concrete examples,
we have already mentioned that (discrete) Clifford groups provide non-trivial examples illustrating
the preceding facts.
More abstractly, given a generalized ring $\Gamma$ (which for simplicity we assume to be
skew), we define a sequence of groups: let $Q_0 := (\Gamma, +)$, 
$$
Q_{1,0} := \Gamma^2=  \Gamma \times \Gamma , \quad
Q_{0,1} := (Q_{1,0})^\vee , \ldots \quad 
Q_{p,q}: = \widehat \times_z^p Q_{1,0} \widehat \times_z (\widehat \times_z^q Q_{0,1} ),
$$
where $\Gamma^2$ is an abelian group, but with non-trivial grading and action
(see Example \ref{ex:central}).
When $\Gamma = \Z / 2 \Z$, then this sequence is exactly the sequence of 
discrete Clifford groups. 
In this case, the first few groups are 

\msk
$Q_0 = (\Gamma,+) = C_2 = \Z / 2 \Z$,

$Q_{1,0} = V = C_2 \times C_2$ (the Klein four-group),

$Q_{0,1} = C = C_4$ (the cyclic group of order $4$),

$Q_{2,0} = Q_{1,0} \wtimes_Z Q_{1,0} = D =D_4$ (the dihedral group of
symmetries of the square),

$Q_{0,2} = Q_{0,1} \wtimes_Z Q_{0,1} = Q$ (the quaternion group of 
order $8$).,

$Q_{1,1} \cong D_4$ (but with a different grading than above).

\msk
\nin
We investigate this {\em discrete Clifford category} in some detail in Section \ref{sec:Cliffordgroups}.
Most facts are, essentially, known from the theory of Clifford algebras: Salingaros denotes these
groups by $G_{p,q}$ and calls them {\em Vee-groups} (see \cite{S, S82}, see also references given in
\cite{A, AM}). 
In the present work, we push the group theoretic approch given by these authors much further,
by putting it in the large framework of graded monoidal categories, allowing much more
general rings than $\Z/2\Z$.  
For instance,
when $\Gamma = \R^2$ with its canonical skew-symmetric product
$x \cdot y = (x_1 y_2 - x_2 y_1, 0)$, 
we get 
a sequence of groups of Heisenberg-type, called the {\em Heisenberg category}.

\subsection{Group algebra, and Clifford algebra}
The group algebra $\K[G]$ of the groups $G = Q_{p,q}$ is a super-algebra which 
 decomposes as a direct sum
of two ideals $\K[G]^+$ (which in turn is a group algebra of an abelian group), and
$\K[G]^-$, which is precisely the Clifford algebra (Theorem \ref{th:Cl} and \ref{th:Super}).
More generally, when $G = G_1 \widehat \times_Z G_2$, then $\K[G]^+$ is the ordinary
(ungraded) tensor product $\K[G_1]\otimes \K[G_2]$ of group algebas, and
$\K[G]^-$ is the {\em graded} tensor product (Theorem \ref{th:Z}).

  \subsection{Concluding remarks}
 It seems as if the category of centrally $\Z/2\Z$-graded sets supports a good deal of
 what can be done on ordinary, ungraded sets. One may ask if this carries
 as far as differential calculus and analysis -- indeed, the joint  paper with J.\ Haut,
\cite{BeH} proposes a setting of ``categorical differential calculus'' which seems to
be suitable for an adaptation to a ``graded'', or ``super'',  framework. 
The important challenge is to understand if, and how, such a calculus were related to
other known categorical approaches (\cite{Mol, Sa, Sch}).
We hope to investigate this question in future work.

Another aspect, already present in, e.g.,  \cite{A, AM, L,S}, 
becomes clearified in the present setting, namely, the relation with the
theory of {\em central extensions of groups}: indeed, 
all our groups can be viewed as central extensions, and as such, 
can be described by certain {\em cocycles} 
(Section \ref{sec:cocyle}).
It becomes apparent that these cocycles are all related to algebra and combinatorics
of {\em totally ordered (finite) index sets}, in particular, via   {\em inversions} (\href{https://de.wikipedia.org/wiki/Fehlstand}{\em Fehlstand}, in German). 
$$
***
$$

\bigskip
{\em
Je souhaite d\'edier ce travail \`a la m\'emoire de mon coll\`egue Lionel B\'erard-Bergery
(1945 -- 2019).
Les discussions que nous avions, autour de sujets de g\'eom\'etrie, o\`u  les graduations 
\'etaient toujours pr\'esentes, 
m'ont \'et\'e une source in\'epuisable d'inspiration dans ce domaine. 
Bien que Lionel n'ait pas eu le temps  de formaliser et de finaliser ces id\'ees, j'esp\`ere qu'elles
verront le jour, t\^ot ou tard, et que ce travail y contribuera.}

\section{ $\Gamma$-graded sets and monoids}\label{sec:1}


In this section, $(\Gamma,+)$ is an abelian group.

\subsection{Additively and multiplicatively graded sets, and grading auto\-mor\-phism}
The definition of {\em (additively) $\Gamma$-graded sets and monoids} has been given above
(Def.\ \ref{def:add-graded} and \ref{def:graded-monoid}).
As explained in the introduction, we'll focus on {\em additively} graded structures, and simply
call them ``graded''. Let's turn them into a category.
We have the choice between a
 ``strong'' or ``weak'' definition of morphisms 
(preserving, or permuting, the grading):

\begin{definition}\label{def:graded-morphism}
Let $X,X'$ be sets and $(\Gamma,+)$ an abelian group.
A {\em morphism of additively $\Gamma$-graded sets $(X,d),(X',d')$, of degree $d(f) \in \Gamma$},
 is a map $f:X \to X'$ such that
$$
\forall x \in X: \qquad
d(f(x))=  d(f) + d(x) .
$$
When $d(f)=0$, we say that
$f$ is a {\em grading-preserving morphism}.

A {\em morphism of centrally $\Gamma$-graded sets} is a grading-preserving
morphism  commuting
with the $\Gamma$-action ($f(\gamma.x)= \gamma.f(x)$, for all
$x\in M$ and $\gamma \in \Gamma$).

A {\em morphism of centrally $\Gamma$-graded monoids or groups} is
a morphism of groups or monoids satisfying the preceding condition.
\end{definition}

\begin{definition}\label{def:alpha}
On every centrally $\Gamma$-graded set $(M,d,z)$, we define
the {\em grading automorphism} by
$$
\alpha:M\to M, \quad x \mapsto \alpha(x) = d(x).x = z(d(x),x) .
$$
\end{definition}

\nin
Clearly, $\alpha$ is a grading-preserving morphism.
 Its inverse is $\alpha^{-1}(x)=z(-d(x) ,x)$.

\begin{definition}\label{def:alpha2}
We say that $(M,-d)$ is the {\em negative grading} of $(M,d)$.
\end{definition}

 \begin{lemma}\label{la:grading-auto}
 If $(G,d,Z)$ is a centrally $\Gamma$-graded group, then
 $\alpha:G\to G$, $g \mapsto Z^{d(g)} g$ is an automorphism of graded groups
 (again called {\em grading automorphism}).
 \end{lemma}
 
 \begin{proof} 
 $\alpha(gh) = Z^{d(gh)} gh = Z^{d(g)} Z^{dh} gh = Z^{d(g)} g Z^{d(h)} h = \alpha(g) \alpha(h)$,
 $\alpha(e)=e$.
 \end{proof}
 
 When $Z:\Gamma \to G$ is injective, then the grading can be recovered from 
 $\alpha$ via $G_\gamma = \{ x \in G \mid \forall  u \in \Gamma : u x = \gamma x \}$
 (a kind of ``pre-eigenspace''). 
 As already noticed, examples of centrally $\Gamma$-graded sets arise from centrally
 graded groups:

\begin{example}\label{ex:P(n)}
Let
 $M: = \cP(\sfn)$ the power set of $\sfn = \{ 1,\ldots , n \}$.
 It is an abelian group for the symmetric difference $A \Delta B$ of sets, and from elementary
 set theory we have a group
 morphism
 $$
 d : \cP(\sfn) \to \Gamma:=\Z/2 \Z, \quad A \mapsto \vert A \vert \mod 2 .
 $$
So $M_0 =\cP(\sfn)_0$ is the subset of sets having even cardinality, and
$M_1=\cP(\sfn)_1$ the one of subsets having odd cardinality. 
A central action of $\Z/2\Z$ on $\cP(\sfn)$ is obtained by choosing some set
$B \subset \sfn$ of even cardinality and letting
$z(\ol 0 ,A) = A$, $z(\ol 1,A,) = A \Delta B$,.
We have
 $\alpha(A) = A$ if $A$ is even, and
$\alpha(A) = A \Delta B$ if $A$ is odd. 
\end{example}

\begin{example}
All dihedral groups $D_n$ are $\Z/2\Z$-graded, via the quotient morphism
$D_n \to D_n / C_n = C_2$ (where $C_n$ is the cyclic group of order $n$).
But only for $n$ even they are {\em centrally} $\Z/2\Z$-graded: only in this case
the center is non-trivial, $Z(D_{2n})= \{ \pm 1 \} \cong C_2$, so that $Z : \Z / 2 \Z \to D_{2n}$
can be taken to be $Z(\ol 0) = 1, Z(\ol 1) = -1$.
Note that the groups $D_{4n}$ admit several normal subgroups of index $2$, containing the element
$-1$, so there
are several different central gradings on $D_{4n}$.
\end{example} 

\begin{example}
The quaternion group $Q$ is centrally $\Z/2\Z$-graded: similarly to the case of $D_4$,
 its center $\pm 1$ is isomorphic to $C_2$, and it has several normal
subgroups of index $2$ containing the element $-1$.  
In contrast to the $D_4$-case, the different normal subgroups in question are 
conjugate to each other under outer automorphisms.
\end{example}

\begin{remark}[Central extensions]\label{rk:central-extension}
The quotient (orbit set) of the $\Gamma$-action,
$M^0:= M / \Gamma $, 
is again $\Gamma$-graded: the map $d$ passes to the
quotient $\overline d : M^0 \to \Gamma$.
If the action of $\Gamma$ is free, it will sometimes be useful to choose
a set-theoretic section of the projection $M \to M^0$, and thus
to identify $M$ with $M^0 \times \Gamma$.
In particular,  the setting of centrally graded groups self-dual in the sense that, reversing
arrows,  $d$ and $Z$   change their roles, and
we have exact sequences of groups
$$
\begin{matrix}
G_0 & \to & G & \to & \Gamma, \\
\Gamma & \to & G & \to & G^0:=G/Z^\Gamma ,
\end{matrix}
$$
the second of them realizing $G$ as {\em central extension} of
$G^0$ by $\Gamma$. Choosing a set-map $G^0 \to G$, section to the
projection $G \to G^0$, the set $G$
can be identified with  $G^0 \times \Gamma$, with group law 
expressed via a map $\tau:G^0 \times G^0 \to \Gamma$
(with $G^0$ written additively),
$$
(h,u) \cdot (h',u') = (hh', u + u' + \tau(h,h')) .
$$
Associativity is equivalent to saying that
 $\tau$ satisfies the {\em cocycle relation}
$$
\tau(h,h') + \tau(hh',h'') = \tau(h,h'h'') + \tau(h',h'').
$$
 (cf.\ Section \ref{sec:cocyle}).
Note that $d$ induces a morphism $d:G^0 \to \Gamma$,
but the induced $[Z]:\Gamma \to G^0$ would be trivial.
 \end{remark}

\subsection{Monoidal structure on the category of centrally
$\Gamma$-graded sets}
We define the product $M=M_1 \times_z M_2$ of two centrally $\Gamma$-graded sets
as explained in Theorem  \ref{th:monoidal-category}.
By routine check, $M$ is again centrally $\Gamma$-graded. (Note that the equivalence
relation is the orbit relation  of  the action of the
``antidiagonal subgroup''
$ \{ (\gamma,-\gamma) \mid \gamma \in \Gamma \}$ on
$M_1 \times M_2$, and the grading passes to the quotient.)
To complete the proof of Theorem \ref{th:monoidal-category}, we check associativity 
$$
(M_1 \times_z M_2) \times_z M_3 = M_1 \times_z (M_2 \times_z M_3).
$$
Indeed,  this follows from associativiy (which we consider
to be  ``strict'') of the usual Cartesian product $\times$,  which passes to the quotient.
Namely,  on both sides, the grading is induced by
$d_1 \oplus d_2 \oplus d_3$ (by associativity and commutativity
of the group $(\Gamma,+)$),
and on both sides, the quotient is taken with respect to the
``generalized antidiagonal''
$$
N:= \{ (\gamma_1,\gamma_2,\gamma_3)\in \Gamma^3 \mid \,
\gamma_1 + \gamma_2 + \gamma_3 = 0 \} .
$$
Thus both sides are identified with
the quotient for the equivalence relation 
$$
z^\gamma [x_1,x_2,x_3] = [z_1^\gamma x_1,x_2,x_3] =
[x_1,z_2^\gamma x_2,x_3] =
[x_1,x_2,z_3^\gamma x_3].
$$
Finally,  the set $E=\Gamma$ with trivial grading
$d=0$ and action on itself by $z^\gamma(x)=\gamma+x$
is neutral for the graded product (direct check).
 And, given morphisms $f_i:M_i \to M_i'$, $i=1,2$, 
 of degree $u_i = d(f_i)$,
 the following is a well-defined morphism from $M_1 \times_z M_2$ to $M_1' \times_z M_2'$ of degree $u_1 + u_2$:
$$
(f_1 \times_z f_2 )[x_1,x_2] =  [f_1(x_1),f_2(x_2)] .
$$
Summing up,
we get a bifunctor, turning centrally $\Gamma$-graded sets with 
their graded product into a 
\href{https://ncatlab.org/nlab/show/monoidal+category}{\em (strict)  monoidal category}
(see \cite{CWM}, Chapter VII, p. 161).

\section{Braidings, and graded product of groups}

Next, we discuss ``commutativity'': what is the relation between $M_1 \times_z M_2$ and
$M_2 \times_z M_1$?
A {\em braiding} on a monoidal category is the choice of a ``choherent'' family of isomorphisms
switching the order of product (see \cite{CWM}, Chapter XI).
In general, a braiding is an {\em additional} structure.
The standard braiding from set theory, $A \times B \cong B \times A$, induces of course 
a braiding on centrally $\Gamma$-graded sets: we call this again the {\em standard braiding}. 

\subsection{Braiding of centrally graded sets}
In order to define other braidings,
from now on, we
assume that $\Gamma$ is a generalized ring,
with bi-additive
product map $\Gamma^2 \to \Gamma$, $(a,b) \mapsto ab$,
 and define the {\em braiding map}
$\beta_{M_1,M_2}:M_1 \times_Z M_2 \to M_2 \times_Z M_1$ as in Definition
\ref{def:braiding}. 
Let us prove Theorem \ref{th:braiding}.
Clearly, $\beta_{M_1,M_2}$ preserves gradings and commutes with
the $\Gamma$-action (since $(\Gamma,+)$ is commutative), hence
is a morphism of centrally $\Gamma$-graded sets.
It is an isomorphism, since it admits an inverse
$$
\beta^{-1} : 
\, [y_1,y_2] \mapsto Z^{-d(y_1)d(y_2)}[y_2,y_1].
$$
Note that this inverse corresponds to the original braiding map
when replacing $d$ by $-d$ and the ``product'' $ab$ of 
$\Gamma$ by its ``opposite product'' $ba$.
Thus, if $ab= - ba$, we have the symmetry condition
$\beta_{M_1,M_2}^{-1}=\beta_{M_2,M_1}$, as claimed.

Saying that our strict monoidal category is 
\href{https://ncatlab.org/nlab/show/braided+monoidal+category}{\em
braided} amounts to the
following two ``hexagonal diagram identities'' (see \cite{CWM}, Chapter XI, p. 253):
first,
the following two morphisms from
$M_1 \times_Z M_2 \times_Z M_3$ to
$M_2 \times_Z M_3 \times_Z M_1$ coincide
(for better readability, we shall write,
$\beta_{1,2} $ instead of $ \beta_{(M_1,M_2)}$,
and $\beta_{1,23}$ instead of
$\beta_{(M_1, M_2 \times_z M_3)}$ , etc.):
$$
(\id_{2} \times_Z \beta_{(1,3)}) \circ
(\beta_{(1,2)} \times_Z \id_{3} )=
\beta_{(1, 23)}.
$$
We prove this by direct computation
(writing  $d$ instead of $d_i$)
\begin{align*}
(\id_{2} \times_Z \beta_{(1,3)} ) \circ
(\beta_{(1,2)}  \times_Z \id_{3} )
[x_1,x_2,x_3] & = 
\\
(\id_{2} \times_Z \beta_{(1,3)} ) 
(Z^{d(x_2)d(x_1)} [ x_2,x_1,x_3] )
& = 
z^{d(x_2)d(x_1)+ d(x_3) d(x_1)} [x_2,x_3,x_1] ,
\\
\beta_{(1, 23)}[x_1,x_2,x_3] & =
z^{(d(x_2) +d( x_3) )d(x_1)} [x_2,x_3,x_1] .
\end{align*}
By distributivity of the product in $\Gamma$, both sides
agree. 
The second ``hexagonal'' condition  amounts to the first 
condition, when replacing the braiding maps by their inverses.
But, as noticed above, the inverses of the braiding are of the same
form (replacing $d$ by $-d$ and the product in $\Gamma$ by its
opposite), so the computation is exactly the same.

\begin{remark}
As explained in \cite{CWM}, Section XI.5 (Braided Coherence),
the preceding arguments amount to proving that
the morphisms
$s_i:=\beta_{(M_i,M_{i+1})}$ satisfy the defining {\em relations of the
generators of the braid group}
$$
s_i s_{i+1} s_i = s_{i+1} s_i s_{i+1}, \mbox{ and }
s_i s_j = s_j s_i \mbox{ whenever } \vert i - j \vert \not= 1.
$$
Explicitly, let's show that
$s_2 s_1 s_2 = s_1 s_2 s_1$, by 
 a computation essentially equivalent to the one given above.
We write $(23)$ for $s_2$ and $(12)$ for $s_1$:
 \begin{align*}
(23)(12)(23)[x_1,x_2,x_3] & =
(23)(12) Z^{d_3 d_2} [x_1,x_3,x_2] \\
& = (23) Z^{d_3 d_2 + d_3 d_1 } [x_3,x_1,x_2] \\
& = Z^{d_3 d_2 + d_3 d_1 + d_2 d_1}[x_3,x_2,x_1]
\\
(12)(23)(12) [x_1,x_2,x_3] & = 
(12)(23) Z^{d_2 d_1 } [x_2,x_1,x_3] \\
& = (12) Z^{d_2 d_1 + d_3 d_1}[x_2,x_3,x_1] \\
& = Z^{d_2 d_1 + d_3 d_1 + d_3 d_2}[x_3,x_2,x_1]
\end{align*}
When the braiding is symmetric, the action on the braid group passes to an action
of the symmetric group $\s_n$ (coherence theorem in symmetric monoidal categories,
see \cite{CWM}, p. 253, Thm.1). 
We give an explicit formula describing this action in terms of
 \href{https://de.wikipedia.org/wiki/Fehlstand}{permutation inversions}: 
\end{remark}

\begin{theorem}\label{th:braided2}
Assume the product in $\Gamma$ is skew. Then  there are well-defined
isomorphisms, for every $\sigma \in \s_n$,
$$
\beta_\sigma^{(M_1,\ldots,M_n)} :
M_1 \times_Z \ldots \times_Z M_n \to
M_{\sigma(1)}\times_Z \ldots \times_Z M_{\sigma(n)} ,
$$
coinciding for adjacent transpositions $s_i =(i,i+1)$ with the braiding
isomorphisms, and in general given by
$$
\beta_\sigma^{(M_1,\ldots,M_n)} [x_1,\ldots,x_n] =
Z^u [x_{\sigma(1)},\ldots,x_{\sigma(n)}],
$$
with $u$ given in terms of {\em permutation inversions} for $\sigma$:
$$
u = \sum_{(i,j) : \atop i<j, \sigma(i) > \sigma(j)} d(x_j) d(x_i).
$$
\end{theorem}

\begin{proof}
Existence of the action follows
from the coherence theorem in symmetric monoidal categories (loc cit.)
All that remains to be proved is the explicit formula for $u$ in terms of 
 inversions.
Indeed, for
$n=2$ and $n=3$, the explicit formulae given above prove the claim.
For general $n$, the claim is given by induction.
To this end, one may decompose any permutation into a product of
adjacent transpositions, each of which adding an inversion term (which may
cancel out with another term obtained precedingly, and so really counts
the inversions:
this is where skewness of the product in $\Gamma$ becomes crucial), and performing a computation
which is essentially the same as the one given above for the permutation
$(13)=(12)(23)(12)$. 
\end{proof}

\subsection{Braided dual of a centrally graded group}
We continue to assume that $\Gamma$ is a generalized ring.
The braiding maps show up in constructions on graded groups (or monoids)
from Theorem \ref{th:graded-group}.
We prove Item (1) of this theorem: the product
 of $G^\vee$  is associative: 
\begin{align*}
g_1\cdot^\vee (g_2 \cdot^\vee g_3) & 
= Z^{d(g_1) \, d(g_2 g_3) } (g_1 \cdot (g_2 \cdot^\vee g_3 )) 
\\
& =
Z^{d(g_1) d(g_2) + d(g_1) d(g_3)  + d(g_2)d(g_3)} g_1 g_2 g_3
\\
&= Z^{d(g_1 g_2) \, d(g_3) }  (g_1 \cdot^\vee g_2)  \cdot g_3
= (g_1 \cdot^\vee g_2) \cdot^\vee g_3 .
\end{align*}
The neutral element is $e$, and the inverse of $g$ is
$Z^{(d(g)^2)} g^{-1}$, since $d(g^{-1})= - d(g)$,
$$
g \cdot^\vee Z^{(d(g)^2)} g^{-1} =
Z^{d(g) \cdot d(g^{-1})} Z^{(d(g))^2} gg^{-1} =
Z^{-(d(g))^2 + (d(g))^2} e = e = Z^{(d(g)^2)} g^{-1} \cdot^\vee g.
$$
Since $\Gamma$ is commutative and $Z(\Gamma)$ central, the maps
$d$ and $Z$ define on this group the structure of a $\Gamma$-graded
group.
Since  the product in $(G^\vee)^\vee$ is
$(g_1,g_2) \mapsto
Z^{d(g_2) d(g_1) + d(g_1)d(g_2)} g_1 g_2$, it  follows that
$(G^\vee)^\vee = G$ if $\Gamma$ is skew.

\begin{remark}\label{rk:assoc-0} 
By  induction, the product of $n$ elements in $G^\vee$ is given by
$$
g_1 \cdot^\vee \ldots \cdot^\vee g_n = Z^u g_1 \cdots g_n, \, \mbox{ where }
u = \sum_{(i,j) : i<j} d(g_i) d(g_j) .
$$
Note that there are also other ``duals'': the {\em opposite group} $G^\opp$ of a group
$G$, and the braided duals belonging to the products $-ab$ and to $ba$, and any
combination of these operations. We are mainly interested in the case $\Gamma =
\Z/2 \Z$, where most of them coincide, and will not investigate systematically the
family of ``duals'' thus obtained.
\end{remark}

\begin{example}[Central extensions of $\Gamma$]\label{ex:central} 
Let $G = \Gamma^2$ with group law here written
$$
(x_0,x_1) \cdot (y_0,y_1) = (x_0 + y_0,x_1 + y_1),
$$
 and 
$d(x_0,x_1) = x_1$ and $Z(x)= (x,0)$.
This clearly defines a (commutative) centrally $\Gamma$-graded group,
denoted by $\Gamma_{1,0}$.
Its $b$-dual, which we denote by
$\Gamma_{0,1} = \Gamma_{1,0}^\vee$,
 is given by the same $d$ and $Z$, and
product
$$
(x_0,x_1) \cdot^\vee (y_0,y_1) =
Z(x_1 y_1) \cdot (x_0+y_0, x_1 + y_1) =
(x_0 + y_0 + x_1 y_1, x_1 + y_1) .
$$
Indeed, $\Gamma_{1,0}$ is the central extension of $\Gamma$
by $\Gamma$ via the trivial cocycle $0$, whereas $\Gamma_{0,1}$ is
the extension via the (in general) non-trivial cocycle $(x,y) \mapsto xy$.

For later use, let us have a closer look at the case  $\Gamma = \Z/2\Z = \{ \ol 0,\ol 1\}$ 
with its usual (field) product.
Let 
$$
1:= e_0 := (\ol 0,\ol 0), \quad Z:= (\ol 1,\ol 0),\quad
e_1 := (\ol 0,\ol 1), \quad
Ze_1 := (\ol 1,\ol 1) .
$$
In the following,
we'll write the group law multiplicatively (although at this stage it will still be commutative).
So let $t \in \{ 1, Z \}$.
When $t=1$, the following gives the group law of $\Gamma_{1,0}$, and when
$t=Z$, it gives the group law of $\Gamma_{0,1}$:
 \begin{center}
\begin{tabular}{|*{7}{c|}}
\hline
 $\cdot$       & $1$ & $Z$ &  $e_1$ &
       $ Z e_1 $
  \\
  \hline\hline
  $1$  & $1$ &  $Z$ &  $e_1$ &
       $Z e_1$
  \\
   \hline
  $Z$  &
  $Z$ &  $1$ & $Ze_1$ &
       $ e_1$
  \\
   \hline   \hline
  $e_1 $   & $e_1 $ & $Ze_1$ & $t $
  & $tZ $
  \\
   \hline
     $Ze_1 $  & $Ze_1 $ & $e_1$ & $tZ$ & $t$
  \\
    \hline
  \end{tabular}
 \end{center}
The elements $1$ and $Z$ are even, and $e_1$ and $Ze_1$ are odd.
For $t=1$, we see that
$\Gamma_{1,0} \cong C_2 \times C_2$ as group (usual direct product).
For $t=Z$, 
we  have $\Gamma_{0,1} \cong C_4$, because $e_1$   becomes
here an element of order four.
\end{example}

\subsection{Graded product of centrally graded groups}\label{ss:graded}
We define the {\em graded product} of two centrally $\Gamma$-graded groups $G_i$,
$i=1,2$, as in Item 5)2) of Theorem \ref{th:graded-group}.
Let us show that the new product is associative:
\begin{align*}
(f_1,f_2 )\cdot \bigl( (g_1,g_2) \cdot (h_1 , h_2) \bigr) &=
(f_1,f_2) \cdot
(Z_1^{d(g_2) \, d(h_1) } g_1h_1, g_1 h_2)
\\
& = (Z_1^{d(f_2) d(g_1h_1) } Z_1^{ d(g_2) d(h_1)} f_1g_1h_1,f_2 g_2 h_2)
\\
&=(Z_1^{d(f_2) d(g_1) + d(f_2) d(h_1) + d(g_2) d(h_1)} f_1g_1h_1,f_2 g_2 h_2)
\\
& =(Z_1^{d(f_2)d(g_1) + d(f_2 g_2) d(h_1)} f_1 g_1 h_1,f_2 g_2 h_2 )
\\
&=
(Z_1^{d(f_2)d(g_1)} f_1 g_1 , f_2 g_2)  \cdot (h_1,h_2)
\\
& = 
\bigl( (f_1,f_2) \cdot (g_1,g_2) \bigr)  \cdot (h_1 , h_2) .
\end{align*}
The 
 inverse element of $(g_1,g_2)$ is, as is directly checked,
$(Z_1^{d(g_2) d(g_1)} g_1^{-1},g_2^{-1})$, and so 
$(G_1 \times G_2,\cdot)$ is a group.
Moreover, $d(g_1,g_2):= d(g_1) + d(g_2)$ clearly defines a morphism
from this group to $(\Gamma,+)$.
The elements $(Z_1^k,Z_2^{-k})$ belong to the center of this group,
and their degree is zero, hence the quotient $G/N$ is again a graded group.
The equivalence relation generated by $N$
comes from $(g_1,g_2) \sim (Z_1^k g_1, Z_2^{-k} g_2)$, that is,
$(g_1,Z_2^k g_2) \sim (Z_1^k g_1,g_2)$ for all $k$, 
and hence the quotient group
$G_1 \times G_2 / N$ fulfills the conditions  from the theorem.
Moreover, if the product in $\Gamma$ is skew,
the braiding map is a group isomorphism
from $G_1 \widehat \times_Z G_2$ onto
$G_2 \widehat \times_Z G_1$:
\begin{align*}
\beta([g_1,g_2]\cdot [h_1,h_2]) & =
\beta ([ Z_1^{d(g_2) d(h_1)} g_1 h_1,g_2 h_2 ])
\\
& = Z^{d(g_2) d(h_1) - d(g_2 h_2) d(g_1 h_1)}
[g_2 h_2,g_1 h_1]
\\
& = Z^{ d(g_2) d(h_1) - (d(g_2) + d( h_2)) (d(g_1) + d( h_1))}
[g_2 h_2,g_1 h_1]
\\
& = Z^{- d(h_2) d(g_1) - d(h_2) d(h_1) - d(g_2) d(g_1) }
[g_2 h_2,g_1 h_1]
\\
&= Z^{d(g_1) d(h_2)} 
[Z_2^{d(g_1) d(g_2) + d(h_1)d(h_2)} g_2 h_2,  g_1 h_1]
\\
& = [Z_2^{d(g_1) d(g_2)} g_2,g_1] \cdot [Z_2^{d(h_1)d(h_2)} h_2,h_1]
\\
 & = \beta  [g_1,g_2] \cdot \beta[h_1,h_2] 
\end{align*}

\begin{remark}[Product of $n$ elements]\label{rk:assoc-1}
With different notation, the formula from the proof
for the product of three elements
reads, with $d_{ij}:=d(g_{ij})$,
$$
(g_{11},g_{12}) \cdot
(g_{21},g_{22}) \cdot
(g_{31},g_{32}) =
Z_1^{d_{12} d_{21} + d_{12}d_{31} + d_{22} d_{31}}
(g_{11}g_{21}g_{31}, g_{12}g_{22}g_{32}).
$$
By induction, the product of $n$ elements in $G$ is given by
$$
[g_{11},g_{12}] \cdots [g_{n1},g_{n2}] =
Z^{\sum\limits_{(i,j)\in \sfn^2 : \atop i<j} d_{i2} d_{j1}} \bigl[g_{11}\cdots g_{n1}, \, g_{12}\cdots g_{n2} \bigr].
$$
Note that the proof of associativity and the proof of the braiding property
are essentially equivalent. 
\end{remark}

To finish the proof of Theorem \ref{th:graded-group}, we prove Item (3):
the set-theoretic equality 
$(G_1  \widehat \times G_2) \widehat \times G_3 = G_1  \widehat \times (G_2  \widehat \times G_3)$ is
an isomorphism of groups: on the one hand,
\begin{align*}
\bigl( (g_1,g_2 ),g_3 \bigr) \cdot \bigl( (h_1,h_2),h_3 \bigr) & =
\bigl(
Z^{d(g_3) \, d(h_1,h_2)} (g_1,g_2)\cdot(h_1,h_2), g_3 h_3 \bigr)
\\
& = \bigl( Z_1^{ d(g_3) (d(h_1) + d(h_2))} Z_1^{d(g_2)d(h_1)} 
g_1 h_1, g_2 h_2,g_3 h_3 \bigr)
\\
& = \bigl( Z_1^{ d(g_3) d(h_1) + d(g_3) d(h_2) + d(g_2)d(h_1)} 
g_1 h_1, g_2 h_2,g_3 h_3 \bigr).
\end{align*}
On the other hand, using distributivity in the ring $\Gamma$, we get the same result:
\begin{align*}
\bigl( g_1,(g_2 ,g_3) \bigr) \cdot \bigl( h_1,(h_2,h_3) \bigr) & =
\bigl(Z_1^{d(g_2,g_3) \, d(h_1)} g_1 h_1, (g_2,g_3)\cdot (h_2,h_3)\bigl) 
\\
& = \bigl(Z_1^{(d(g_2) + d(g_3)) \, d(h_1)} Z_1^{d(g_3)d(h_2)}
g_1 h_1, g_2 h_2,g_3 h_3 \bigr)
\\
& =  \bigl( Z_1^{ d(g_2) d(h_1) + d(g_3) d(h_1) + d(g_3) d(h_2)}
g_1 h_1, g_2 h_2,g_3 h_2 \bigr)  
\end{align*}

\begin{remark}\label{rk:assoc-2}
The preceding computation shows that the group product in
$G_1 \widehat \times_Z G_2 \widehat \times_Z G_3$ is given by,
with $d_{ij}:= d_{(i,j)}:=d(g_{ij})$,
$$
[g_{11},g_{12} ,g_{13}] \cdot [g_{21},g_{22},g_{23}]=
Z^{ d_{12} d_{21} + d_{13} d_{21} + d_{13} d_{22}}
[g_{11} g_{21},g_{12} g_{22} ,g_{13} g_{23}]  .
$$
By induction, we get the product of two elements in
$G=G_1 \widehat \times_Z G_2 \widehat \times_Z \ldots \widehat \times_Z 
G_p$,
$$
[g_{11},\ldots,g_{1p}] \cdot [g_{21},\ldots,g_{2p}] 
=
Z^{\sum_{i>j} d_{1i} d_{2j}} [g_{11} g_{21},\ldots,g_{1p} g_{2p}].
$$
Note the ``duality'' with the formula from Remark \ref{rk:assoc-1}!
Applying it twice, we get
\begin{align*}
[g_{11},g_{12} ,g_{13}] \cdot [g_{21},g_{22},g_{23}] \cdot
[g_{31},g_{32},g_{33}] & =
Z^u  [g_{11} g_{21} g_{31} \,,g_{12} g_{22} g_{32} \,,g_{13} g_{23} g_{33}] , 
\end{align*}
with the degree term given by 
\begin{align*}
u & =
d_{12} d_{21} + d_{13} d_{21}  + d_{13}d_{22} +
(d_{12}+d_{22}) d_{31} +
(d_{13}+ d_{23}) d_{31} +
(d_{13} + d_{23}) d_{32}
\\ 
& = (d_{12} d_{21} + d_{13} d_{21} + d_{13} d_{22}) +
(d_{12} d_{31} + d_{13} d_{31} + d_{13} d_{32}) +
(d_{22} d_{31} + d_{23} d_{31} +  d_{23} d_{32}) .
\end{align*}
The $9$ terms are structured by the graph of the usual total order
on $\{ 1, 2, 3 \}$ (the
 three pairs
$(1<2),(1<3),(2<3)$).
By induction (the details of which we omit), we get for the product of $n$
elements 
$g_i = (g_{i1},\ldots,g_{ip})$, for $i=1,\ldots,n$:
$$
g_1 \cdots g_n = Z^u 
(g_{11}\cdots g_{n1},\ldots, g_{1p}\cdots g_{np})  
$$
with degree term given by a kind of ``matrix product''  
$$
u = {\sum_{(i,j,k,\ell) \in {\mathsf p}^2 \times \sfn^2 
:\atop  i>j, k<\ell } d_{ki} d_{\ell j}} .
$$
 \end{remark}

\begin{remark}
The group inverse in $G_1 \widehat \times_Z G_2 \widehat \times_Z G_3$ 
is given by
$$
[g_1,g_2,g_3]^{-1} = Z^{d_2 d_1 + d_3 d_1 + d_3 d_2}[g_1^{-1},
g_2^{-1},g_3^{-1}].
$$
By induction, $[g_1,\ldots,g_n]^{-1} =
Z^{\sum_{i<j} d_j d_i }[g_1^{-1},\ldots,g_n^{-1}]$.
\end{remark}


\begin{remark}[Choices and conventions]\label{rk:conventions}
The graded product of graded groups could have been defined by other choices and
conventions. To explain this, note that we could have defined the group law on
$G_1 \times_z G_2$ in a different way:  every matrix $X \in M(2,2;\Z)$ gives rise to 
a bi-additive product on $\Gamma \oplus \Gamma$, via
$$
\langle  a, b \rangle = (a_1,a_2) X (b_1,b_2)^t
= a_1 X_{11} b_1 + a_1 X_{12} b_2 + a_2 X_{21} b_1 + a_2 X_{22} b_2 
$$
(where products in $\Gamma$ are given by the fixed bi-additive product, and coefficients
from $\Z$ associate with everything). Thus we can  define
a braided dual group structure on $G_1 \times G_2$, as above, with respect to the
bi-additive product on $\Gamma \oplus \Gamma$, and then mod out the normal subgroup
$N$ as before. The result is a new group structure on $G_1 \times_z G_2$, depending on
the matrix $X$.
When $\Gamma = \F_2 = \Z / 2 \Z$, then all $16$ matrices from $M(2,2;\F_2)$ give rise
to a group structure.
For instance, the matrix $X = E_{11}$ gives rise to the group
$G_1^\vee \times_z G_2$, and $X = I = E_{11}+E_{22}$ gives rise to the group
$G_1^\vee \times_Z G_2^\vee$.
However, for these choices, Item (3) of Theorem \ref{th:graded-group} would fail to hold:
for this we need that $X$ is strict (upper or lower) triangular.
For $X = 0$, this gives the ungraded product; our choice is
$X = E_{12}$, but $X = E_{21}$ would give an isomorphic theory.
\end{remark}

\begin{remark}[Internal graded product]\label{rk:internal}
With notation as in Theorem \ref{th:graded-group}, $G$ has two normal subgroups
\begin{align*}
\{ [g_1,1] \, \mid \, g_1 \in G_1 \} & \cong G_1,
\\
\{ [1,g_2]\, \mid \, g_2 \in G_2 \} & \cong G_2,
\end{align*}
such that 
$G_1 \cap G_2 = Z^\Gamma = \im(Z)$.
Commutators of elements of $G_1$ with elements
of $G_2$  belong to $Z^\Gamma$. Indeed, both groups normalize each other:
\begin{align*}
[1,g_2]\cdot [g_1,1] \cdot [1,g_2]^{-1} & = Z^{d(g_2)d(g_1)} [g_1,g_2] \cdot [1,g_2^{-1}]
\\ & =
[Z_1^{d(g_2)d(g_1)}  g_1,1] = Z^{d(g_2)d(g_1)}  [g_1,1] .
\end{align*}
These facts imply  that $G$ is a certain 
\href{https://en.wikipedia.org/wiki/Free_product#Generalization:_Free_product_with_amalgamation}{amalgamated semi-direct product} of $G_1$ with $G_2$.
Conversely, such data can be used to recover the graded 
product (so we may speak of an ``internal graded product'' of subgroups
of a given group, corresponding to the ``external graded product'', given
by the theorem).
Thus we may write
$g_1 g_2$ instead of $[g_1,g_2]$, and compute products by using the
``commutation relation''
\begin{equation}\label{eqn:internal}
g_2 g_1 = Z^{d(g_2) d(g_1)} g_1 g_2 .
\end{equation}
\end{remark}

\begin{example}[$D_4$ and $Q$ revisited]\label{ex:D4-table}
To prepare the grounds for the following section, 
let us compute explicitly the graded product of two groups of type
$\Gamma_{1,0}$ or $\Gamma_{0,1}$, as defined in Example \ref{ex:central}.
Let $G_i = \{ 1 , Z_i , e_i , Z_i e_i \}$, given by the table from Example \ref{ex:central},
with $t_i \in \{ 1, Z_i \}$, where $Z_i$ corresponds to $Z$.
The cardinality of $G=G_1 \widehat \times_Z G_2$ is
$\frac{4 \times 4}{2}  = 8$, and its elements are
$$
1 = [(1,1)], \quad
Z = [(1,Z_2)]=[(Z_1,1)], \quad
e_1 = [(e_1,1)], \quad
e_2 = [(1,e_2)],
$$
$$
Z e_1,  \quad
Z e_2, \quad
e_{12}:=e_1 e_2 = [(e_1,e_2)], \quad
Z e_{12} .
$$
These elements are multiplied by using the rules: $e_1 e_2 = Z e_2 e_1$,
$e_i^2 = t_i$, $Z_i^2 = 1$, $Z$ is central. 
The elements $1,Z,e_{12},Ze_{12}$ 
are even, and $e_1,e_2,Ze_1,Ze_2$ are odd. 
For instance $e_{12}^2 = e_1 e_2 e_1 e_2 = Z e_1^2 e_2^2 = Z t_1 t_2$.
The group table of $G$ is as follows:

\ssk

 \begin{center}
\begin{tabular}{|*{9}{c|}}
\hline
 $\cdot$       & $1$ & $Z$ & $e_{12}$ & $Z e_{12} $ & $e_1$ &
       $ Z e_1 $ & $e_2$ & $Z e_2$ 
  \\
  \hline \hline
   $1$       & $1$ & $Z$ & $e_{12}$ & $Z e_{12} $ & $e_1$ &
       $ Z e_1 $ & $e_2$ & $Z e_2$ 
  \\
   \hline
  $Z$  &
  $Z$ &  $1$ & $Ze_{12}$ &
       $ e_{12}$ & $Z e_1$ & $e_1$ & $Z e_2$ & $e_2$ 
       \\ \hline
   $e_{12}$ & $e_{12}$ & $Z e_{12}$ & $Z t_1 t_2$ & $t_1 t_2$ & $Z t_1 e_2 $ & $t_1 e_2$ & $t_2 e_1$ & $Zt_2 e_1$
   \\ \hline
   $Z e_{12}$ & $Z e_{12}$ & $e_{12}$ &
   $ t_1 t_2$ & $Z t_1 t_2$ & $ t_1 e_2 $ & $Z t_1 e_2$ & $Z t_2 e_1$ & $t_2 e_1$
  \\
   \hline   \hline
  $e_1 $   & $e_1 $ & $Ze_1$ & $t_1 e_1 $
  & $Z t_1 e_1 $ & $t_1$ & $Z t_1$ & $e_{12}$ & $Z e_{12}$ 
  \\
   \hline
     $Ze_1 $  & $Ze_1 $ & $e_1$ & $Z t_1 e_1$ & $t_1 e_1$ & $Z t_1$ & $t_1$ & $Z e_{12}$ & $e_{12}$
    \\
    \hline
   $e_2$ & $e_2$ & $Z e_2$ & $Z t_2 e_1$ & $t_2 e_1$ & $Z e_{12}$ & $e_{12}$ & $t_2$ & $Z t_2$
   \\ \hline
   $Z e_2$ &$Ze_2$ & $e_2$ & $ t_2 e_1$ & $Zt_2 e_1$ & $ e_{12}$ & $Ze_{12}$ & $Zt_2$ & $ t_2$
   \\
   \hline
  \end{tabular}
 \end{center}

\msk
\nin
Since $G$ is non-commutative and of cardinality $8$, it is either isomorphic to the dihedral group
 $D_4$ or to
the quaternion group $Q$.
When $t_1 = Z = t_2$, then $G$ contains $6$ elements of order $4$, and hence is 
isomorphic to  $Q$. Its even subgroup is $G_0 \cong C_4$.
In the other cases, it contains exactly $2$ elements of order $4$, and hence is isomorphic to $D_4$.
The even subgroup then is isomorphic to $C_4$ if $t_1 = 1 =t_2$, and to
$C_2 \times C_2$ if $t_1=Z,t_2=1$, or if  $t_1=1,t_2=Z$.
\end{example}

\begin{example}
Let $m \geq 2$, and $\Gamma = \Z/m\Z$.  We have three groups, of cardinal
$m^3$,
$$
C_m^2 \widehat \times_Z C_m^2, \qquad
C_{m^2} \widehat \times_Z C_m^2, \qquad
C_{m^2} \widehat \times_Z C_{m^2} .
$$
It follows from Remark \ref{rk:internal}  that these groups satsify
$[G,G]\subset \langle Z \rangle$, hence are 
$2$-step nilpotent. 
E.g., for $m=3$, this gives the
\href{https://en.wikipedia.org/wiki/List_of_small_groups#List_of_small_non-abelian_groups}{two non-abelian groups of
order $27$}
(see \cite{D}  p. 565 and p. 566 for detailed information on these groups).
\end{example}

\section{The discrete Clifford category}\label{sec:Cliffordgroups}

\subsection{The $\Gamma$-category}
Assume the product in
$\Gamma$ is skew, and recall from Example
\ref{ex:central} 
the graded groups 
$$
Q(1):=\Gamma_{1,0} = \Gamma^2,
\qquad
Q(Z):= \Gamma_{0,1}
= \Gamma_{1,0}^\vee.
$$
The symmetric monoidal category of $\Gamma$-graded groups
generated by $Q(1)$ and $Q(Z)$ will be called the
{\em $\Gamma$-category}.
Explicitly, these groups are of the  form:

\begin{definition}
For $\ttt = (t_1,\ldots,t_n) \in \{ 1 ,Z \}^n$, 

if $t_i=1$, let $Q(t_i)$ be a copy of $Q(1)=\Gamma^2$,

if $t_i = Z$, let $Q(t_i)$ be a copy of $Q(Z)=Q(1)^\vee$,
and define 
$$
Q(\ttt) := Q(t_1) \widehat \times_Z \ldots \widehat \times_Z Q(t_n) .
$$
We let also, 
$Q_0 := \Gamma$ with trivial grading
(as we have seen, this is the neutral element for the graded product),
and
for  $p+q=n$,
$$
Q_{p,q} := Q(\underbrace{1,\ldots,1}_{p \times},\underbrace{Z,\ldots,Z}_{q \times}),
\quad
Q_n := Q_{n,0} = Q(1,\ldots,1), 
\quad
Q_n^\vee = Q_{0,n}.
$$
\end{definition}

By Item (3) of Theorem \ref{th:graded-group}, we can recast this definition:

\begin{definition}
Let $K$ be a set. 
For $(\ttt,\sss) \in K^n \times K^m$, we denote by 
$$
\ttt \oplus \sss := (t_1,\ldots,t_n,s_1,\ldots,s_m) \in K^{n+m}
$$
their ``juxtaposition''.
Note that this operation is associative, but not commutative (in fact,
it is the composition law of the 
\href{https://en.wikipedia.org/wiki/Free_monoid}{free monoid over $K$} ).
\end{definition}

\begin{lemma}
For $\ttt \in \{ 1,Z \}^n$, $\sss \in \{ 1,Z\}^m$,  we have
$$
Q(\ttt \oplus \sss) = Q(\ttt) \widehat \times_Z Q(\sss).
$$
In particular, $Q_{p,q} = Q_p \widehat \times_Z Q_{0,q}$.
\end{lemma}

\begin{remark}
Fom Item (4) of Theorem
\ref{th:graded-group}, we get an isomorphism
$\phi_{(\ttt,\sss)}:Q(\ttt \oplus \sss) \to Q(\sss \oplus \ttt)$.
For instance, re-arranging the order of factors, we get an isomorphism
$$
Q_{p,p} \cong Q_{1,1}  \widehat \times_Z \ldots  \widehat \times_Z
Q_{1,1}.
$$
\end{remark}

\begin{example}
Let $\Gamma = \R^k$ with $k \geq 2$ and skew-symmetric product
$(x_1,x_2) = (0,\ldots,0,x_1 x_2 - x_2 x_1)$. Then $Q_{1,0}$ is the abelian group
$\R^{2k}$, whereas $Q_{0,1}$, and all higher $Q_{p,q}$ are 
 are $2$-step nilpotent Lie groups,  isomorphic to a 
\href{https://en.wikipedia.org/wiki/Heisenberg_group}{\em Heisenberg
group}.
We propose to call the $\Gamma$-category, in this case, the
{\em Heisenberg category}.
A systematic study of this category is certainly an interesting topic for subsequent work. 
\end{example}

\subsection{The discrete Clifford category}
From now on, we assume that $\Gamma = \Z/2 \Z$  with its usual ring (field) structure. 
In this case, the $\Gamma$-category will be called the {\em (discrete)
Clifford category}, since the groups
$Q_{p,q}$ then are (discrete) Clifford groups, as we shall see.
The specific feature of this case is that $\Gamma$ is both skew and admits a unit element
$\ol 1$ (which is not the case in the Heisenberg category).
To 
give an explicit description of all groups by generators and relations, we follow the notation from
Example \ref{ex:D4-table}:
for $m \in \N$, let $Z_m$ be a copy of the element $Z$ (copy called
``of $m$-th generation'').
For each $t_m \in \{ 1,Z_m\}$,
 we define a graded group of four  elements,
$$
Q(t_m) = \{ e_0 , e_m , Z_m , Z_m e_m \},
$$
just as in Example \ref{ex:D4-table}:
$1 = e_0$ is neutral, 
 $Z_m^2 = e_0$, and
$e_m^2 = t_m e_0$.
The elements
$e_i, Z_i e_i$ are odd, and $e_0,Z_i$ are even. 
Thus, as we have seen in Exemple \ref{ex:central},
\begin{equation}
\begin{matrix}
Q_1 = Q(1) =C_2^2,\\
\quad Q_{0,1} = Q(Z) = C_4 .
\end{matrix}
\end{equation}
From Example \ref{ex:D4-table}, we now get
\begin{equation}
\begin{matrix}
Q_2 &= & Q(1,1) &= &D_4, \\
Q_{0,2} &= &Q(Z,Z) &=& Q, \\
Q_{1,1} & = &Q(1,Z) &=& D_4 .
\end{matrix}
\end{equation}
As already said,  $Q(\ttt)$ will be seen to be isomorphic to the discrete Clifford group of
the Clifford algebra $\Cl(\ttt)$ (Appendix \ref{app:Clifford}).
In the following, let us prove some basic structure results on these groups, without using
the theory of Clifford algebras. 
Recall from Example \ref{ex:P(n)} the abelian group $(\cP(\sfn),\Delta)$,
isomorphic to $(Z/2 \Z))^n$.  

\begin{theorem}\label{prop:relations}\label{la:shuffle}\label{th:shuffle} 
Fix $\ttt \in \{ 1,Z \}^n$. Then:
\begin{enumerate}
\item
The group $G:=Q(\ttt)$ is of cardinality $2^{n+1}$.
\item
The quotient group $G^0 = G / \{ 1 , Z \}$ is abelian, isomorphic to
$\cP(\sfn)=(\Z / 2 \Z)^n$, i.e., 
the following is an exact sequence of groups (central extension)
$$
\begin{matrix}
\{ 1, Z \} & \to & Q(\ttt) & \to&  (\cP(\sfn),\Delta) .
\end{matrix}
$$
Elements of $Q(\ttt)$ are of order either $1$, $2$ or $4$.
\item
The group $G$ is generated by the elements $e_i \in Q(t_i)$ (which are identified with the corresonding
element in $G$). For $A \subset \sfn$  with $k = \vert A \vert$, 
whose elements are ordered,
$a_1 < \ldots < a_k$, we let
$$
e_A := e_{a_1} \cdots e_{a_k}  .
$$
Then $G$ is a disjoint union, as follows:
$$
G = \{ e_A \mid A \in \cP(\sfn) \} \, \sqcup \, \{ Z e_A \mid A \in \cP(\sfn) \} .
$$
We often identify the first set with $\cP(\sfn)$ (i.e., we fix this set theoretic splitting of the
exact sequence from Item (2)). 
 There is a function
$\tau:\cP(\sfn) \times \cP(\sfn) \to \Z / 2 \Z$ such that, for all $A,B \in \cP(\sfn)$,  
$$
e_A e_B = Z^{\tau(A,B)} e_{A \Delta B}.
$$
\item
The group $G$ is $2$-step nilpotent:
the commutator subgroup $[G,G]$ belongs to the center of $G$,
whence  $[G,[G,G]]= \{ 1 \}$.
\item
Defining relations between the generators $e_1,\ldots,e_n,Z$ 
are: 
\begin{enumerate}
\item
$e_0$ is neutral, 
\item
 (for all $i=1,\ldots,n$) $e_i^2 = t_i e_0$, 
 \item
 (for $k < \ell$)
$e_k e_\ell = Z e_\ell e_k$, and 
\item
$Z$ is central, even, and of order $2$. 
\end{enumerate}
\end{enumerate}
\end{theorem}

\begin{proof}
(1), (2), (3) are essentially contained in Theorem \ref{th:graded-group} and
Remark \ref{rk:assoc-2}. 
More formally, these items are proved by
induction: for $n=1$ and $n=2$, this has been noticed above.
Assume the claim already proved at level $n \in \N$, and write
\begin{equation}\label{eqn:recursion}
Q(t_1,\ldots,t_{n+1}) = Q(\ttt) \widehat \times Q(t_{n+1}) ,
\end{equation}
and use ``internal notation'' (\ref{eqn:internal}) together with
induction hypothesis to write elements of this group as claimed.
Concerning the product of two elements,
$e_A e_B = e_{a_1} \cdots e_{a_k}e_{b_1} \cdots e_{b_\ell}$,
re-order elements in ascending order of indices; when two indices coincide,
the square yields either $1$ or $Z$; only terms with indices from
the symmetric difference survive, together with a term $Z^u$ with
$u=0$ or $u=1$.
It follows that $(e_A)^4 = 1$ and $(Z e_A)^4 = 1$, for all
$A \in \cP(\sfn)$.

(4) 
The morphism from item (2) sends the commutator
group $[G,G]$ to $[\cP(\sfn),\cP(\sfn)] = 0$ since
$\cP(\sfn)$ is abelian. Therefore
$[G,G]\subset \{ 1,Z \}$, and hence $[G,G]$ belongs to the center of $G$. 
(See Lemma \ref{la:AB} for an explicit formula describing commutators.) 

(5) Relations (a) -- (d) are contained in the preceding items. 
Conversely, given a group having generators satisfying these 
relations, it is seen by induction that it is isomorphic to $Q(\ttt)$:
for $n=1$ and $t_1 =1$, we have the defining relations of $C_2 \times C_2$,
and for $n=1$ and $t_1 = Z$, we have the defining relations of $C_4$;
likewise,
for $n=2$, we have the defining relations of $D_4$, resp., of $Q$.
For the induction step, assume the claim holds at order $n$, and let
$G_{n+1}$ be a group having generators and relations of the given form at
order $n+1$. Then $H:=\{ 1, e_{n+1} , Z , Z e_{n+1} \}$ is a subgroup,
isomorphic to $C_4$ if $t_{n+1} = Z$ and to $C_2^2$ if $t_{n+1}=1$,
and the given relations show that $G$ is a homomorphic image
of $G_n \widehat \times_Z H$. For reasons of cardinality, $G$ is actually
equal to this group, i.e., given by 
(\ref{eqn:recursion}).
\end{proof}

\begin{remark}
Following the pattern of ``classification of real Clifford algebras'', we can now
``classify'' the groups $Q_{p,q}$. In our opinion, this gives a transparent and
conceptual version of that part of Clifford theory.
Since, in the main text, we are rather interested in classification-free theory,
we relegate the presentation of this issue to Appendix 
\ref{app:classification}.
\end{remark} 

\begin{remark}
The function $\tau$ is the cocycle of a central extension.
See Section  \ref{sec:cocyle} for more on this;
cf.\ also \cite{L} p.12--14.
 \end{remark}

Next, we compute an explicit formula for the inner automorphism given by $e_A$:

\begin{lemma}\label{la:AB}
For all $A,B \in \cP(\sfn)$,
$$
e_A \cdot e_B \cdot (e_A)^{-1} = Z^{\vert A \vert \cdot \vert B \vert - \vert A \cap B \vert}
\, e_B .
$$
\end{lemma}

\begin{proof}
When $\vert A \vert = 1 = \vert B \vert$, then the claim is in keeping with
the defining relations of $Q(\ttt)$, namely
$e_i e_j (e_i)^{-1} = u e_j$, with $u=1$ if $i=j$, and $u=Z$ else.
Now let $\vert A \vert = 1$, say, $A = \{ i \}$, and $B$ arbitrary.
When $i \notin B$, then we have to exchange the position of $e_i$
and each $e_j$, $j \in B$, which gives a factor $Z^{\vert B \vert}$,
and when $i \in B$, we get a factor $Z^{\vert B \vert - 1 }$ (since
we exchange for all $i \in B$ with $i\not= j$), which again is in keeping
with the claim. For general $A$ and $B$, conjugation by $e_A =
e_{a_1} \cdots e_{a_k}$ is composition  of conjugation by the $e_{a_i}$.
We count the total number of exchanges of elements
$e_i$ and $e_j$ arising in the whole procedure: whenever $(i,j) \in
A \times B$ with $i \not=j$, we get an additional factor $Z$, so in the end
we get a factor $Z^p$ with
$$
p = \vert \{ (i,j) \in A \times B \mid i \not= j \} \vert 
= \vert A \times B \setminus {\rm diag} \vert
= \vert A \vert \cdot \vert B \vert - \vert A \cap B \vert ,
$$
as claimed.
\end{proof}

\begin{example}\label{ex:pseudo}
For the {\em pseudoscalar} (case $A = \sfn$, $A \cap B = B$),
$e_\sfn = e_1 \cdots e_n$,
we get:
$$
e_\sfn e_B (e_\sfn)^{-1} = Z^{(n-1) \vert B \vert} e_B ,
$$
which is equal to $e_B$ if $n$ is odd, and equal to
$Z^{\vert B \vert} e_B$ if $n$ is even.
\end{example}

\begin{theorem}[Center, inner automorphisms]\label{th:center} 
Let $\ttt \in \{  1,Z \}^n$ and $G:=Q(\ttt)$.
Assume $n$ is even. Then:
\begin{enumerate}
\item
the center of $G$ is the subgroup
$\{ 1, Z \}$, 
\item
 there are $2^n + 1$ conjugacy classes:
two of them are singletons, $\{ 1 \}$ and $\{ Z \}$, and the
other are of cardinal two, $\{ e_A, Z e_A \}$, for
$A \not= \emptyset$,
\item
the inner automorphism group
of $Q(\ttt)$ is isomorphic to $\cP(\sfn) \cong (\Z / 2 \Z)^n$,
\item
the grading automorphism $\alpha(e_A)=Z^{\vert A \vert} e_A$ 
is an inner automorphism. 
\end{enumerate}
Assume $n$ is odd. Then:
\begin{enumerate}
\item
the center of $G$ is the subgroup
$\{ 1, Z, e_\sfn, Z e_\sfn \}$ (when $\ttt = (1,\ldots,1)$, it is
isomorphic to $C_4$ if $n\equiv 1 \mod 4$, and to
$C_2^2$ if $n \equiv 3 \mod 4$), 
\item
there are
$2^n + 2$ conjugacy classes, four of them singletons (elements
of the center), the other $2^n - 2$ classes 
of the
form $\{ e_A, Z e_A \}$,  $A \not= \emptyset, A \not= \sfn$,
\item
the inner automorphism group  is isomorphic to
$\cP(\sfn - 1)  \cong (\Z / 2 \Z)^{n-1}$,
\item
the grading automorphism $\alpha$ is an outer automorphism.
\end{enumerate}
In both cases, $\alpha$, together with the inner automorphisms,
generates
a subgroup of automorphisms isomorphic to $\cP(\sfn) \cong (C_2)^n$.
\end{theorem}

\begin{proof}
From Lemma \ref{la:AB}, $e_A$ commutes with all $e_i$ iff 
$\vert A \vert - \vert A \cap \{ i \}$ is even, for all $i$, that is, 
for all $i$,
the set $A \setminus \{ i  \}$ has an even number of elements.
The only possibility to realize this case is  $A =\sfn = \{ 1,\ldots , n\}$,
where $n$ is odd. 
Since $Z$ is always central, and $e_\sfn^2 = Z^{\frac{n(n-1)}{2}}$,
 this proves the statements about
the center.

Concerning conjugation classes, 
Lemma \ref{la:AB} shows that that each class has at most $2$ elements,
say $e_B$ and $Z e_B$.
It has $1$ element iff the element is central, and $2$ elements else.
Thus the claims follow from those about the center.
Likewise, those on the inner automorphism group $G/Z(G)$ follow
from those on the center.

The grading automorphism from Lemma \ref{la:grading-auto}
satisfies the condition $\alpha(e_k)= Z^{d(e_k)} e_k =
Z e_k$, whence 
$\alpha(e_A) = Z^{\vert A \vert} e_A$.
As seen in Example \ref{ex:pseudo}, when $n$ is even, it is inner.
When $n$ is odd, then there is no set $A \in \cP(\sfn)$ such that
$e_A e_k (e_A)^{-1} = Z e_k $ for all $k=1,\ldots,n$ 
(since $A = \sfn$ does not satisfy the condition, and when
$A \not= \sfn$, the value of $e_A e_k (e_A)^{-1} $
depends whether $k \in A$ or $k \notin A$).
Thus the grading automorphism then is outer.
\end{proof}

In general, the group of {\em all} automorphisms of $Q(\ttt)$ is not
a suitable object, for our purposes. We are interested in the
case that all $t_i$ are equal, because then permutations of the $e_i$
induce automorphisms. But in this case, the group of all automorphisms
may be too big, as is illustrated by the following
examples: 
\begin{example}\label{ex:aut-1}
Let $n=1$, $G=Q(t_1) = \{ e_0,e_1,Z,Ze_1\}$,
$e_1^2 = t_1$.
\begin{enumerate}
\item
When $t_1 = Z$, so $G \cong C_4$, then
$\Aut(G) \cong C_2$ (the unique non-trivial
automorphism exchanges the two elements of order $4$,
namely $e_1$ and $Z e_1$).
\item
When $t_1 = 1$, so $G \cong C_2^2$, then
$\Aut(G) \cong \s_3$ is the permutation group of the three elements
of order $2$, namely of $e_1,Z,Ze_1$.
\end{enumerate}
\end{example}

\begin{example}\label{ex:aut-2}
Let $n=2$, $G = Q(t_1,t_2) = \{ e_0,e_1,e_2,e_{12},Z,
Z e_1, Z e_2,Z e_{12} \}$,
$e_i^2 = t_i$.

\begin{enumerate}
\item
When $t_1 = 1 = t_2$, then $\Aut(G) \cong D_4 \cong G$.
Conjugation by elements $e_1,e_2,e_{12}$ defines three non-trivial
inner automorphisms of order $2$, and exchange of $e_1$ and $e_2$
yields an outer automorphism of order $2$. 
Conjugation by $e_{12}$ commutes with all automorphisms. 
\item
When $t_1 = Z = t_2$, then $G = Q$ and $\Aut(G) \cong \s_4$.
Namely, the inner automorphisms form a Klein $4$-group as above,
and permutations of $i,j,k$ (that is, of $e_1,e_2,e_{12}$),
form a complementary $\s_3$-subgroup.
However, exchange of $e_1$ and $e_{12}$ does not preserve the
grading (since $e_1$ is odd and $e_{12}$ even).
The subgroup of grading-preserving automorphisms is again a
$D_4$-subgroup of $\s_4$, having same types of elements as
in (1). 
\end{enumerate}
\end{example}


\begin{theorem}[Hyperoctahedral automorphism group]\label{th:Bn}
The symmetric group $\s_n$ acts by automorphisms on $Q_n$: 
for every $\sigma \in \s_n$, there is a unique automorphism
$$
\Phi_\sigma : Q_n\to Q_n
\mbox{ such that } \forall k=1,\ldots,n : \,
\Phi_\sigma (e_k) = e_{\sigma(k)} .
$$
This group of automorphisms normalizes the group $(C_2)^n$
described in Theorem \ref{th:center} , and together they form
a group $B_n$ of grading-preserving
automorphisms of cardinality $n! 2^n$.
The grading automorphism $\alpha$ belongs to the center of this
group.
The same statements hold for the group $Q_{0,n}$.
\end{theorem}

\begin{proof}
By induction:
for $n=1$ the statement is uninteresting;
for $n=2$, the braiding automorphism the unique
automorphism
$$
\tau: Q_2 = Q_1 \widehat \times_Z Q_1 \to Q_2, \quad
[g_1,g_2] \mapsto Z^{d(g_1)d(g_2)} [g_2,g_1]
$$
such that
$\tau(e_1) = e_2, \quad
\tau(e_2)=e_1, \quad
\tau(Z)= Z, \quad
\tau(e_{12}) = Z e_{12}
$
(cf.\ also Example \ref{ex:aut-2}).
Assuming that $\s_n$ acts by automorphisms on $Q_n$, we write
$$
Q_{n+1} = Q_n  \widehat \times Q_1 = 
Q_{n-1} \widehat \times Q_2 .
$$
Then $\s_n = \s_n \times \id_{Q_1}$ acts on $Q_{n+1}$, and another
copy of $\s_2 = \id_{Q_{n-1}}$ acts on $Q_{n+1}$
(transposition $(n \, n+1)$). Together, these two actions
generate the action of $\s_{n+1}$.
Similarly for 
$Q_{0,n} = Q_{0,1}\widehat \times \ldots \widehat \times Q_{0,1}$.

It is clear that the action of $\s_n$ normalizes the inner automorphisms
as well as the grading automorphism $\alpha$, hence the group
$(\Z / 2 \Z)^n$ is a normal subgroup, intersecting trivially the
$\s_n$-subgroup. Thus $B_n$ is a semidirect product, of cardinality
$n! 2^n$. 
Clearly, $\alpha$ commutes with the $\s_n$-action, hence belongs
to the center of $B_n$.
\end{proof}

\begin{remark}
Recall that, by definition, the 
\href{https://en.wikipedia.org/wiki/Hyperoctahedral_group}{hyperoctahedral
group} is the group of {\em signed permutation matrices} (of size
$n \times n$). It is Coxeter group of type $B_n$,
generated by the usual permutation matrices (subgroup
$\s_n$) and the diagonal $\pm 1$-matrices (subgroup $C_2^n$).
It is isomorphic to the automorphism group defined in the preceding
theorem.
Its center is equal to $\{ \id , \alpha \}$. 
\end{remark}

\begin{definition}
We call the group $B_n$ of automorphisms defined in the preceding theorem
the {\em hyperoctahedral automorphism group} of $Q_n$, resp.\ of
$Q_{0,n}$.
\end{definition}

\begin{remark}
The whole group of grading preserving automorphisms is in general
bigger than the hyperoctahedral automorphism group.
\end{remark}

\begin{remark}
For $Q_{p,q}$, the symmetry group $\s_n$ must be replaced by
$\s_p \times \s_q$, permuting the $e_i$'s with $t_i = 1$, resp.\ those
with $t_i = Z$.
\end{remark}

\begin{remark}[Inversions] 
In the Clifford category, all algebraic operations can be described by explicit formulae.
From an algebraic point of view, this amounts to describe structures (central extensions)
by {\em cocycles}, and from a combinatorial point of view, the main tool is given by
{\em inversions}. 
For instance, to compute a formula for the automorphism $\Phi_\sigma$ from Theorem \ref{th:Bn},
we start by writing  $A =\{ a_1,\ldots,a_k\}$ with 
$a_1 <\ldots < a_k$, and
$$
\Phi_\sigma(e_A) = \Phi_\sigma(e_{a_1}\cdots e_{a_k})=
e_{\sigma(a_1)} \cdots e_{\sigma(a_k)}.
$$
We have to re-order these terms so that the indices are in increasing
order. The number of terms $Z$ appearing is
$u:=\vert (i,j)\in A \times A \mid \, i<j, \sigma(i) > \sigma(j) \vert$, so
we end up with the formula of ``inversion type''
$$
\Phi_\sigma(e_A) = Z^{(u \mod 2)}  e_{\sigma(A)} . 
$$
\end{remark}

\subsection{Cocycles}\label{sec:cocyle}
Recall that the 
 \href{https://ncatlab.org/nlab/show/group+extension}{central extension}
of a group $H$ by an abelian group $A$ can be described by
a cocyle $C : H \times H \to A$.
In our case, by  Theorem \ref{th:shuffle} we have the central extension
 $A = \{ 1 , Z \} \to Q(\ttt) \to H = \cP(\sfn)$, with cocycle denoted by
 $\tau$, defined by the rule
$$
e_A e_B = \tau(A,B) e_{A \Delta B}.
$$
Expanding both sides of
$(e_A e_B) e_C = e_A (e_B e_C)$,
we get
the {\em cocycle relations}:
\begin{equation}\label{eqn:cocycle}
\tau(A,B) \tau(A\Delta B, C) = \tau(A,B\Delta C) \tau(B,C).
\end{equation}
Since our cocyle $\tau$ depends on $\ttt$ we call it a
``relative cocycle''. We shall separate it into two parts, the first of which
contains the explicit dependence on $\ttt$, and the second being 
independent of $\ttt$ (``absolute'').

\begin{definition}
Let $K$ be a commutative monoid,
$\ttt \in K^n$ and $A \in \cP(\sfn)$. We let
$$
\ttt_\emptyset := 1, \qquad 
\ttt_A := \prod_{a \in A} t_a  \in K.
$$
\end{definition}

\nin
The map $A \mapsto \ttt_A$ behaves like a ``multiplicative measure'':
when 
$A \cap B = \emptyset$, then $\ttt_{A \cup  B} = \ttt_A \cdot \ttt_B$,
and, for general $A,B$,
$$
\ttt_A \ttt_B = 
\ttt_{A \cap B} \ttt_{A \cup B} = \ttt_{A\cap B}^2 \ttt_{A \setminus B}
\ttt_{B \setminus A} = \ttt_{A\cap B}^2 \ttt_{A \Delta B} . 
$$
If $k^2 = 1$ for all $k \in K$ (which is the case for $K = \{ 1,Z\}$),
then
$$
\ttt_{A \Delta B} = \ttt_A \cdot \ttt_B.
$$

\begin{lemma}
Let $\ttt \in \{ 1 , Z \}^n$.
The map $(A,B) \mapsto \ttt_{A \cap B}$ defines a cocyle on
$\cP(\sfn)$.
\end{lemma}

\begin{proof}
$\ttt_{A \cap B} \ttt_{(A \Delta B)\cap C} =
\ttt_{A\cap B} \ttt_{A \cap C}\ttt_{B \cap C} =
\ttt_{A \cap (B \Delta C)} \ttt_{B \cap C}$.
\end{proof}


We call the cocycle from the preceding lemma ``relative''. 
It will be combined with an  ``absolute'' cocyle $\gamma$:

\begin{definition}
Let
$L := \{ (a,b) \in \sfn^2 \mid a<b \}$ be the usual total order 
relation on $\sfn$, and
$L^{-1} = \{ (a,b)\in \sfn^2 \mid a>b\}$ its opposite order.
For a subset $R \subset (\sfn \times \sfn)$, call
$$
\Inv(R) := R \cap L^{-1}= \{ (a,b)\in R \mid a > b \}
$$
the set of {\em inversions in $R$}, and
for $(A,B) \in \cP(\sfn)^2$, let
$$
\gamma(A,B) := 
Z^{ \vert \Inv(A \times B)  \vert } =
Z^{ \vert \{ (a,b) \in A \times B \vert \, a > b\} \vert} .
$$
\end{definition}

\begin{theorem}[The absolute cocycle]\label{th:absolute-cocycle}
Let $\ttt \in \{ 1 , Z \}^n$ and $A,B \in \cP(\sfn)$.
Then the product $e_A e_B$ in $Q(\ttt)$ is given by
$$
e_A e_B = \ttt_{A \cap B} \cdot \gamma(A,B) \cdot e_{A \Delta B} . 
$$
The map $\gamma : \cP(\sfn)^2 \to \{ 1 , Z\}$
is the cocylce defining $Q_n$ as central extension of $\cP(\sfn)$.
\end{theorem}

\begin{proof}
As in the proof of Theorem \ref{th:shuffle}, we re-order the product
$e_A e_B = e_{a_1}\cdots e_{a_k} e_{b_1} \cdots e_{b_l}$. 
After re-ordering,
indices $i \in A \cap B$ give rise to a square $e_i^2 = t_i$, whence
the factor $\ttt_{A \cap B}$.
While re-ordering, a relation of type $e_i e_j = Z e_j e_i$ is applied
when $i > j$, first for $j = b_1$, then for $j=b_2$, and so on, so the
number of times we apply it is
$$
\sum_{b \in B} \vert \{ a \in A \mid a > b \} \vert =
\vert \{ (a,b) \in A \times B \mid a > b \} \vert = \vert \Inv(A,B)\vert .
$$
When $\ttt =(1,\ldots,1)$ 
 (case of $Q_n$), the factor $\ttt_{A\cap B}$ is always $1$, so
$\gamma(A,B)$ is the cocyle of $Q_n$ (and in particular satisfies
again the cocycle relation).
 \end{proof}

Besides satisfying the cocycle relation
(\ref{eqn:cocycle}),
 the cocyle $\gamma$ has several other properties, which in turn
can be used to develop Clifford theory based on the discrete 
Clifford group (see  \cite{L} for a systematic exposition of this viewpoint).



\section{Group algebra, and Clifford algebras}\label{sec:algebra}

In the following, we start to develop some ``harmonic analysis'' on graded groups -- 
far from complete, with the main aim to clarify the r\^ole of the Clifford algebra in 
case of the discrete Clifford category. Remarkably, this shows some analogy with the
\href{https://en.wikipedia.org/wiki/Heisenberg_group#Representation_theory}{harmonic analysis
of the Heisenberg group}. 
This deserves to be developed elsewhere, and in the following we assume that all groups
are discrete, or even finite.

\subsection{Group (super) algebra}
From now on we fix a commutative unital base ring $\K$.
We assume that the scalar $2$ is invertible in $\K$.
For a set $M$, we denote by $\K[M]$ the free $\K$-module
with basis denoted $(\delta_m)_{m \in M}$, 
so elements of $\K[M]$ are finite linear combinations
$\sum_{m\in M} \lambda_m \delta_m$; if $M = G$ is a
semigroup, then $\K[G]$ carries an associative product $*$
defined by $\delta_x * \delta_y = \delta_{xy}$ (semigroup algebra; cf.\ \cite{FH});
if $G$ is a monoid with unit $e$, the basis element $\delta_e$ becomes a
neutral element of $\K[G]$, and if $G$ is a group, then inversion
$G \to G$ induces an anti-automorphism
$t : \K[G]\to \K[G]$, sometimes called 
 the {\em canonical anti-automorphism} or {\em antipode}.
 
 \ssk
 In the following, let $G$ be a centrally $\Z/2 \Z$-graded group, and assume the
 element $Z(\ol 1)$ is non-trivial (we denote it again by $Z$, so that the image
 of $Z : \Z/2 \Z \to G$ is the group $\{ e , Z \}$).
 Let $\pi : G \to G^0= G/ \{ e,Z\}$ be the canonical projection, and
  $\bA := \K[G]$ be the group algebra.
 There are two maps of order two on $G$, which induce two linear maps of
 order two on $\bA$:
 \begin{align}
 Z : G \to G, \, g \mapsto Zg=gZ, & \quad \alpha : G \to G, \, g \mapsto Z^{d(g)} g,
 \\
 Z_*:\bA \to \bA, \, f \mapsto Z_* f, & \quad
 \alpha_*:\bA \to \bA,\, f \mapsto \alpha_* f.
 \end{align}
Since $2$ is invertible in $\K$, both linear maps can be diagonalized; since
$\alpha \circ Z = Z \circ \alpha$, they can be jointly diagonalized.
We denote the eigenspace decompositions by
\begin{align}
\mbox{ for } Z_* : & \quad \bA = \bA^+ \oplus \bA^-, \quad
\bA^\pm =\{ f \in \bA \mid \, Z_* f = \pm f \},
\\
 \mbox{ for } \alpha_* : & \quad \bA = \bA_0 \oplus \bA_1,
 \quad  \bA_k = \{ f \in \bA \mid \, \alpha_* f = (-1)^k f \} .
 \end{align}
Since $Z$ is central, the map $Z_*$ commutes with elements of the group algebra,
and hence the eigenspaces $\bA^\pm$ are {\em ideals} of the group algebra.
Since $\alpha$ is a group automorphism, by functoriality of the group algebra construction,
 it follows that the induced linear map is an {\em algebra automorphism of order two}, and hence
 the eigenspaces would behave multiplicatively if we indexed them by the eigenvalues
 $\pm 1$. However, following the usual convention for {\em super-algebras}, we take
 $0$ and $1$ as indices, so the eigenspaces obey the ``super algebra rules''
 $$
 \bA_i * \bA_j \subset \bA_{i+j \mod 2}.
 $$
 Each of the two ideals $\bA^\pm$ is, in turn, again a superalgebra.

 \begin{theorem}[Eigenspaces of $Z_*$; superalgebras]\label{th:Z}\label{th:Super}
With notation as above:
\begin{enumerate}
\item
The algebra $\K[G]$ is  a direct sum of ideals,
$\K[G]  = \K[G]^+ \oplus \K[G]^-$,
and
both eigenspaces are isomorphic as $\K$-modules. 
The projection $\pi : G \to G^0 = G/\{ e,Z\}$ induces an algbra morphism
$\pi_*:\K[G] \to \K[G^0]$ having kernel $\K[G]^-$ and image isomorphic to $\K[G]^+$, so
$\K[G]^+\cong \K[G^0]$ as algebra.
\item
With respect to $\alpha_*$, the group algebra and the ideals from the preceding item
are super-algebras. 
The disjoint union $G = G_0 \sqcup G_1$ induces a decomposition
$\K[G] = \K[G_0]\oplus \K[G_1]$ such that
$$
\K[G_0] \subset \bA_0, \qquad \bA_1 \subset \K[G_1] .
$$
\end{enumerate}
\end{theorem}
 
 \begin{proof}
(1) 
Let $P^\pm$ be the two projectors onto the eigenspaces. 
To see that both eigenspaces are isomorphic,
choose a set $H$ of representatives for the $\{ e,Z \}$-cosets
in $G$ (i.e., a section of $\pi$), so that 
$G = H \sqcup ZH$ as set, whence
$\K[G] = \K[H] \oplus Z_* \K[H]$ as $\K$-module.
Then 
$\K[G]^+$ represents the diagonal, and $\K[G]^-$ the
 antidiagonal in this decomposition, and
 the restrictions and corestrictions
$$
P^+:\K[H] \to \K[G]^+, \qquad P^-:\K[H] \to \K[G]^- 
$$
are linear isomorphisms, so both $\K[G]^\pm$ are isomorphic to
$\K[H]$ as $\K$-modules. (When $G$ is finite, 
both spaces have dimension
$\vert H \vert = \vert G \vert / 2$).

Since $\pi(Z)=\pi(1)$, it follows that
$\pi_*(1-Z)=0$. Thus the ideal generated by $1-Z$, which is
$P^-(\K[G]) = \K[G]^-$,  belongs to
the kernel of $\pi_*$. 
On the other hand, the basis
$(h + Zh)_{h\in H}$, is sent to the basis
$\pi_*(h+Zh)= 2 \pi_*(h)$ of $\K[H]$, so the restriction of $\pi_*$
to $\K[G]^+$ defines a bijection
$\K[G]^+ \to \K[H]$, whence 
 $\K[G]^-$ is the kernel of $\pi_*$, and $\pi_*$ induces an isomorphism
 of $\K[G]^+$ onto $\K[G^0]$.

\ssk
(2)
We have already noticed that the decompositions define superalgebras. To prove
the remaining statements,
we decompose further $H = H_0 \sqcup H_1$ into even and odd elements,
whence $G = H_0 \sqcup Z H_0 \sqcup H_1 \sqcup Z H_1$.
Since $\alpha$ is the identity on $G_0$, we get
$\K[G_0] \subset \bA_0$.
Also, $P^+ \K[H_1] \subset \bA_0$,
since, if $g \in H_1$, then
 $\alpha_* (\delta_g  + Z_* \delta_g) = \delta_{Zg} + \delta_g = \delta_g + Z_* \delta_g$.
It follows that
$\bA_1 = P^- \K [H_1] \subset \K[G_1]$. 
\end{proof}

\subsection{Graded and ungraded tensor product} 
If $A,B$ are sets, the free module with basis $A \sqcup B$ is
$\K[A \sqcup B] = \K[A]\oplus \K[B]$, and the one with basis $A \times B$ is
$\K[A \times B] = \K[A] \otimes \K[B]$. 
The tensor product of two associative algebras is again an associative algebra, for the
``usual'' product
$a \otimes b \cdot a' \otimes b' =  aa' \otimes bb'$.
If $\bA, \bB$ are $\Z/ 2 \Z$-{\em graded}, then their {\em graded tensor product}
is the tensor product, as module, together with bilinear product 
\begin{equation}\label{eqn:graded-tp} 
a \otimes b \cdot a' \otimes b' = (-1)^{\vert b \vert \, \vert a' \vert} 
aa' \otimes bb',
\end{equation}
where $a,b,a',b'$ are assumed to be {\em homogeneous} elements.
If $\bA$ and $\bB$ are associative, then this is again an {\em associative} algebra
(cf.\ \cite{BtD}, p. 56).

\begin{theorem}\label{th:Cl}
Let $(G_i,d_i,Z_i)$, $i=1,2$, two centrally $\Z/2 \Z$-graded groups, and
$G = G_1 \widehat \times_Z G_2$. Then the algebras $\K[G]^\pm$ are given in terms
of ungraded, resp.\ graded tensor products of algebras by those of $G_i$, $i=1,2$, as
$$
\K[G]^+ = \K[G_1]^+ \otimes \K[G_2]^+, \qquad
\K[G]^- = \K[G_1]^- \widehat{\otimes} \,  \K[G_2]^- .
$$
In particular, it follows that for $G = Q(\ttt)$, the ideal $\bA^-$ is a Clifford algebra:
for $\ttt \in \{ 1,Z \}^n$, define
$\ul \Cl (\ttt)$ to be the Clifford algebra of $\K^n$ for the diagonal quadratic form
with coefficient $1$ if $t_i=1$ and $-1$ if $t_i = Z$. Then 
$$ 
\K[Q(\ttt)]^+ \cong \K[\cP(\sfn)] ,\qquad
\K[Q(\ttt)]^- \cong \ul \Cl(\ttt) .
$$
In particular, $\K[Q_{p,q}]^- \cong \Cl_{p,q}(\K)$.
\end{theorem}

\begin{proof}
Choose $H$ in $G$, and $H_i$ in $G_i$, as in the preceding proofs.
As above, elements
$P^\pm (\delta_h) = \frac{1}{2}(\delta_h \pm Z_* \delta_{h_i})$, $h \in H$, form a basis in  $\K[G]^\pm$,
and likewise such elements with $h_i\in H_i$ form a basis in $\K[G_i]^\pm$.
We compute, in $\K[G]$,
\begin{align*}
\delta_{h_1} \otimes \delta_{h_2} \cdot \delta_{h_1' }\otimes \delta_{h_2'} & =
\delta_{[h_1,h_2]} \cdot \delta_{[h_1',h_2']} \\
& =
Z^{d(h_2)d(h_1')} * \delta_{[h_1 h_1',h_2h_2']} =
Z^{d(h_2)d(h_1')} * \delta_{h_1 h_1' } \otimes \delta_{h_2h_2'  },
\end{align*}
and project with $P^-$ onto $\K[G]^-$.
Since $Z$ projects to $-1$, this gives
\begin{align*}
P^-(\delta_{h_1)} \otimes P^-(\delta_{h_2}) \cdot
P^-(\delta_{h_1'}) \otimes P^-(\delta_{h_2'}) & =
\\
(-1)^{d(h_2)d(h_1')} 
P^-(\delta_{h_1})P^-(\delta_{h_1'})\otimes P^-(\delta_{h_2}) P^-(\delta_{h_2'}),&
\end{align*}
so the algebra product is given by the graded tensor product.
On $\K[G]^+$, since $Z$ projects to $1$, there is no sign change,
and we get the usual tensor product.

The statement on the Clifford agebras follows from the preceding statement by
 induction, using  Relation (\ref{eqn:Clifford-tensorproduct})
for Clifford algebras. 
(Cf. 
the arguments given in loc.cit.,  \cite{BtD}, where this is used to prove that the dimension
of the Clifford algeba $\Cl(\K^n)$ is $2^n$.
On the other hand, knowing the dimension of the Clifford algebra, the proof could also be 
given by noticing that the generators of $\bA^-$ satisfy exactly the same relations as the
generators of the Clifford algebra $\Cl(\ttt)$.)
\end{proof}

\begin{example}[Group algebra of $Q$ and of $D_4$]
The preceding arguments can be used to show, in a transparent way, that the Clifford algebra
$\Cl_{2,0}$ is isomorphic to $M_2(\K)$, and that
$\Cl_{0,2}$ is the usual quaternion algebra. 
Namely,
by the preceding results, the group $Q_{p,q}$ is (isomorphic to) the group generated
by the basis elements $e_i$ in the Clifford algebra $\Cl_{p,q}(\K) \cong \K[Q_{p,q}]^-$.
For instance, the quaternion group $G=Q$ if often defined as group of matrices
$\{ \pm 1 , \pm i, \pm i , \pm k \} \subset \bH^* \subset \Gl(2,\bC)$.
Now, modulo $-1$, these matrices furnish a basis of $\bH$, 
the elements multiply as they should in the group,
and the matrix $Z = -1$ acts as it
should on $\K[G]^-$ (by the scalar $-1$): so we can conclude that
$\K[G]^- \cong \bH$ as algebra, whence $\Cl_{0,2}(\R) \cong \bH$.
The same argument works for the group $G = D_4$, which we realize as matrix
group $\{ \pm 1, \pm I,\pm J, \pm K\} \subset \Gl(2,\Z)$.
Modulo $-1$, the four matrices give a basis of $M_2(\K)$, and $Z = -1$ acts by $-1$,
whence $\K[G]^- \cong M_2(\K)$. 
\end{example}

\begin{remark}
The fact that Clifford algebras are parts of group algebras, or images of them, is known (cf.\ \cite{A, AM, S}),
and we hope the framework of graded groups proposed here clarifies those approaches.
For instance, the abstract ingredients of the
{\em classification of Clifford algebras} are quite neatly featured in the abstract approach 
 (see  Appendix 
\ref{app:Qpq}, cf. also \cite{S82} and
\cite{L}, p.53, for this issue).
\end{remark}

\subsection{Harmonic analysis of the group algebra}
The group $G = Q(\ttt)$ acts on $\K[G]$ in the usual way
(left-regular representation). This representation decomposes into
sub-representations (ideals) 
$\K[G]=\K[G]^+ \oplus \K[G]^-$. We'll decompose it further.
Since $\K$ is just a ring and not a field, we do not speak about
decomposition into ``irreducible'' modules, but rather
determine a basis of the space of {\em class functions} (which
in case of $\K= \bC$ correspond to the irreducible characters). 
Recall that a {\em class function} is just an element
$f \in \K[G]$ commuting with every $g \in G$, that is, an element of
the center of $\K[G]$.

Since $\K[G]^+ \cong \K[\cP(\sfn)]$, and the group is abelian and each
element is of order at most two, a complete decomposition of this algebra
is easy:

\begin{proposition}
The group characters of $\cP(\sfn)$ are  the maps,
for $A \in \cP(\sfn)$,
$$
\chi_A : \cP(\sfn) \to \{ \pm 1 \}, \quad
B \mapsto \chi_A(B) := (-1)^{\vert A \cap B\vert}.
$$
Thus $\K[\cP(\sfn)] \cong \oplus_{A \in \cP(\sfn)} \K \chi_A$
is a decomposition into one-dimensional ideals. 
\end{proposition}

\begin{proof}
Since $(A\Delta B) \cap C = (A \cap C)\Delta (B \cap C)$, and 
$
\vert A \Delta B \vert  
\equiv  \vert A \vert + \vert B \vert \mod(2)$,
$$
\chi_A (B \Delta C) = (-1)^{\vert A \cap (B \Delta C) \vert} =
(-1)^{\vert A \cap B\vert} (-1)^{\vert A \cap C\vert}=
\chi_A(B) \chi_A(C),
$$
and $\chi_A(\emptyset)= (-1)^0 = 1$,
so $\chi_A$ is a morphism. Since the map
$$
\chi : \cP(\sfn) \times \cP(\sfn) \to \{ \pm 1 \}, \quad
(A,B) \mapsto \chi_A(B) = (-1)^{\vert A \cap B\vert}
$$
is symmetric, it follows that
$$
\chi_\emptyset = 1 , \qquad
\chi_{A \Delta B} (C) = 
\chi_A (C) \chi_B(C),
$$
so $A \mapsto \chi_A$ is a group morphism from $\cP(\sfn)$ to its dual group.
Its kernel is trivial: 
$\chi_A(X) = 1$ for all $X$ means that
$\vert A \cap X \vert$ is even, for all $X$, and taking for $X$ the 
singletons, if follows that $A$ contains no element, so
$A =\emptyset$. 
Thus the morphism $A \mapsto \chi_A$ is injective.
Since characters are always linearly independent, 
the cardinality of $\widehat H$ cannot exceed the one of $H$,
and thus the morphism is a bijection of $\cP(\sfn)$ onto its dual group.
\end{proof}


\begin{remark}
Parametrizing the characters as in the proposition, the
{\em character table} of $G = \cP(\sfn)$ is the square matrix
 $$
( \chi_A(C))_{(A,C) \in \cP(\sfn)^2} = 
 ((-1)^{\vert A \cap C \vert})_{(A,C) \in \cP(\sfn)^2}.
$$
A specific feature of these matrices is that they are {\em symmetric}: 
 $\chi_B(C)=\chi_C(B)$.
Explicitly, for $n=1$, $\cP(\sfone) = \{ \emptyset, \sfone \}$,
resp.\ $n=2$, $\cP(\sftwo) = \{ \emptyset, \{ 1 \}, \{ 2 \}, \sftwo \}$,
we get
$$
\begin{pmatrix} 1 & 1 \\ 1 & - 1 \end{pmatrix},
\qquad
\begin{pmatrix} 1 & 1 & 1 & 1 \\
1 & - 1 & 1 & -1 \\
1 & 1 & - 1 & -1 \\
1 & - 1 & -1 & 1 
\end{pmatrix}.
$$
\end{remark}

\begin{theorem}\label{th:HAn}
Let $\ttt \in \{ 1,Z\}^n$ and  $G = Q(\ttt)$, and 
$\bA := \K[G]^G$ the algebra of central $\K$-valued functions on $G$
(center of $\K[G]$).
Define for
$A \in \cP(\sfn)$, 
$$
E_A^+ := (e_0 +Z) \sum_{B \in \cP(\sfn)} (-1)^{\vert A \cap B\vert} e_B .
$$
Then $\K[G]^+ \subset \K[G]^G$, and
$(E_A)^+$ with $A \in \cP(\sfn)$ is a basis of $\K[G]^+$ such that
(orthogonal idempotents)
Moreover,
\begin{enumerate}
\item
Assume $n$ is even.
Then a basis of $\K[G]^G$ is given by the elements $(E_A^+)_{A\in \cP(\sfn)}$,
along with $e_0^- := (1-Z)e_0$.
The element $e_0^-$ is a basis of the center of  the 
Clifford algebra $\K[G]^-$.
\item
Assume $n$ is odd.
 Then a basis of $\K[G]^G$ is given by the elements 
 $(E_A^+)_{A\in \cP(\sfn)}$,
along with $e_0^- := (1-Z)e_0$ and $e_\sfn^- := (1-Z)e_\sfn^-$ (the
pseudoscalar).
The elements $(e_0^-,e_\sfn^-)$ form a basis of the center of
the Clifford algebra $\K[G]^-$.
\end{enumerate}
If $\K$ is a field, then each basis element corresponds to an irreducible representation of
$G$, and each irreducible $G$-representation arises in this way.
\end{theorem}

\begin{proof}
The dimension of $\K[G]^G$ is equal to the number of conjugacy classes
of $G$, since the characteristic functions of conjugacy classes form a
basis. By Theorem \ref{th:center}, we know the number of conjugacy classes.
We have already exhibited $2^n$ linearly independent central functions;
thus one (resp.\ two) additional independent
 elements suffice for defining a basis. 
Since $\K[G]^G$ is $Z$-invariant, it decomposes into $Z$-eigenspaces,
so we can choose these additional elements in $\K[G]^-$, i.e., in the
Clifford algebra. 
These elements being central, and since $\K[G]^+ * \K[G]^- = 0$,
it is necessary and sufficient that these elements  belong to the center
of $\K[G]^-$.
Now, it is obvious that $e_0^-$, resp.\ $e_\sfn^-$ (for $n$ odd),
satisfy these conditions, and for reasons of dimension, there are
no other independent elements satisfying them.
This proves (1) and (2).
If $\K$ is a field, then by the general theory of representations of finite groups
(see \cite{FH}) every irreducible $G$-representation arises in $\K[G]$. The elements
$E_A^+$ define one-dimensional (hence irreducible) representations; the other elements
generate representations that may decompose into direct sums of several irreducibles
(cf.\ the following remark). 
\end{proof}

\begin{remark}
To get more information,  in case $n$ is odd, 
a further distinction should be made, 
 according to the square of the pseudoscalar being $1$ or $Z$.
 When all $t_j=1$, this corresponds to the cases
$n \equiv 1, 5$, resp.\  $n \equiv 3,7$ modulo $8$.
This determines whether the two submodules of dimension bigger than one
are isomorphic to each other, or not 
(cf.\ \cite{BtD}, p. 288, see also \cite{FH}, Exercise 3.9, p. 30).
\end{remark}

\appendix 

\section{Clifford algebras}\label{app:Clifford}

\subsection{General definitions}
Let $\K$ be a commutative base ring in 
which $2$ is invertible. 
To every  quadratic space $(V,q)$ is associated, in a functorial way,
a unital associative
 algebra $\Cl(V,q)$, its {\em Clifford algebra}, together with a
linear map $j:V \to \Cl(V,q)$, such that 
$$
\forall v \in V : \quad (j(v))^2 = q(v)1 ,
$$
and which is universal for this property,
see e.g.,  \cite{BtD, L}. 
We shall be interested in the case
$V = \K^n$, together with a quadratic form $q = q_\ttt$, where 
$\ttt = (t_1,\ldots,t_n) \in \K^n$,
  $$
  q_\ttt(x_1,\ldots,x_n) = \sum_{i=1}^n  t_i x_i^2, 
 $$
 that is, the canonical basis $e_1,\ldots,e_n$ is orthogonal, and
$q(e_i) = t_i$.
The corresponding Clifford algebra  is then denoted by
$$
\Cl_\K(\ttt) = \Cl_\K(t_1,\ldots,t_n) = \Cl(V,q_\ttt).
$$
When $\ttt = {\bf 0} =(0,\ldots,0)$, this is the exterior algebra: 
$\Cl(V,q_{\bf 0})= \wedge(\K^n)$.
When $\ttt = (1,\ldots,1,-1,\ldots,-1)$, with
$p$ terms $1$ and $q$ terms $-1$, we write
$$
\Cl_{p,q}(\K):= \Cl(\ttt) = \Cl( 1,\ldots,1,-1,\ldots,-1).
$$
Among these algebras, $\Cl(1,\ldots,1) = \Cl_{n,0}(\K)$ plays a special
role: it is {\em anchored}.

\begin{lemma}
For every $\ttt \in \K^n$, the linear map 
$\K^n \to \K^n$, $x \mapsto (t_1 x_1,\ldots,t_n x_n)$
induces an algebra morphism
$$
\Cl_\K(t_1^2,\ldots,t_n^2) \to \Cl_\K(1,\ldots,1) .
$$
\end{lemma}

\begin{proof}
The linear map  is a morphism of quadratic spaces from
$(\K(t_1^2,\ldots,t_n^2))$ to  $(\K,(1,\ldots,1) )$, and  induces,
by the universal property, an algebra morphism.
\end{proof}

We call {\em anchor} the morphism defined in the lemma, and say that
the Clifford algebra $\Cl_\K(t_1^2,\ldots,t_n^2) $ is {\em anchored}.
Of course, when $t_i \in \{ \pm 1 \}$, then the only anchored Clifford algebra
is $ \Cl_\K(1,\ldots,1)$. 
The terminology is motivated by \cite{BeH}: indeed, the anchored Clifford algebras should
be the correct ``super-analog'' of the anchored tangent algebras defined in loc.cit.,
and ``differential calculus'' then is the theory describing the contraction for $\ttt \to 0$
(which in the graded case still remains to be worked out). 

\subsection{Clifford basis}
The standard basis $e_1,\ldots,e_n$ of $\K^n$ defines $n$ independent
 elements in the Clifford algebra $\Cl_\K(\ttt)$.
For a subset $A = \{ a_1 , \ldots , a_k \} \subset \sfn$, with
$a_1 < \ldots < a_n$, we define the element
$$
e_A := e_{a_1} \cdots e_{a_k} \in \Cl_\K(\ttt).
$$

\begin{theorem}
The elements  $(e_A)_{A\in \cP(\sfn)}$ form a $\K$-basis of the
Clifford algebra $\Cl_\K(\ttt)$, called its {\em Clifford basis}.
The Clifford product of two basis elements is given by 
$$
e_A e_B = (-1)^{m(A,B)}  \prod_{k \in A \cap B} t_k  \cdot 
 e_{A \Delta B},
$$
where $e_\emptyset = 1$, 
 $A\Delta B$ is the symmetric difference of $A,B \subset \sfn$, and 
$$
m(A,B) = 
\vert  \{ (a,b) \in A \times B \mid \, a > b \} \vert .
$$
\end{theorem}

\begin{proof}
The dimension of the Clifford algebra is $2^n$
(cf.\ \cite{BtD}), and the elements
$e_A$ generate it, hence form a basis. 
To compute the product $e_A e_B$, using the Clifford rules 
$$
e_k^2 = t_k =t_k e_\emptyset,\qquad
\forall k\not= \ell : \,
 e_k e_\ell = - e_\ell e_k,
$$
we see that $e_A e_B$ is a multiple of $e_{A \Delta B}$. 
Terms with indices $k=\ell \in A \cap B$ give rise to the scalar
$ \prod_{k \in A \cap B} t_k $.
Terms with indices $(k,\ell) \in A \times B$ 
have to be exchanged  when $k > \ell$ (giving rise to a sign change),
and are not altered when $k < \ell$
(the arguments are the same as those of Theorem  \ref{th:absolute-cocycle}).
\end{proof}

Cf.\   page 12 - 14 of \cite{L} for a similar presentation of the preceding
result.

\section{Classification of Clifford groups and algebras}\label{app:classification}

\subsection{Graded and ungraded product}
Both the classification of real Clifford algebras and the one of discrete Clifford groups proceeds
by decomposing them into {\em ungraded} tensor products of elementary algebras, resp.\ products
of groups. On the level of algebras, we have, e.g.,
$\Cl_\R(1,-1,-1) \cong M_4(\bC) \cong \bC \otimes M_2(\R) \otimes M_2(\R)$.
On the level of groups, recall the ungraded product of groups
$G_1 \times_Z G_2$ (quotient of $G_1 \times G_2$ under $(x_1,Zx_2) \sim (Zx_1,x_2)$, written
also $G_1 \times_{C_2} G_2$ when $\Gamma = \Z/2 \Z$).
We'll see that all groups $Q(\ttt)$ can be decomposed as ungraded product 
$\times_{C_2}$ of groups of type
$$
C_2, \quad C_2 \times C_2, \quad C_4 , \quad D=D_4, \quad Q ,
$$
corresponding to the decomposition of Clifford algebras into tensor products involving
$$
\R, \quad \R \oplus \R, \quad \bC, \quad M_2(\R), \quad \bH .
$$ 
The key lemma for classifying discrete Clifford groups is:

\begin{lemma}\label{la:Q-iso}
For all $n \geq 3$ and $\ttt \in \{ 1,Z \}^n$, there is a group
isomorphism 
$$
Q(\ttt) \, \cong \,   Q(t_1,t_2) \times_Z Q( Z t_1 t_2 t_3,\ldots, Z t_1 t_2 t_n).
$$
\end{lemma}

\begin{proof}
Elements of $Q(\ttt)$ are either of the form
$e_A$ or $Z e_A$,  with $A \in \cP(\sfn)$.
When $A \subset \{ 1 , 2 \}$, we get the $8$ elements of the first factor
$Q(t_1,t_2) = \{ e_0,e_1,e_2,e_{12}, Z ,Z e_1,Z e_2,Z e_{12} \}$.
Next, let
$H \subset Q(\ttt)$ be the subgroup generated by the elements
$$
b_3:=e_{1} e_2 e_3, \quad b_4:= e_{1} e_2 e_4, \quad  \ldots, \quad b_n:=  e_{1} e_2 e_n 
$$
Elements of $Q(t_1,t_2)$ and $H$ commute: indeed,
for all $k=3,\ldots,n$, 
$$
e_1 b_k = e_1 e_1 e_2 e_k =
Z^2 e_1 e_2 e_k e_1 = b_k e_1, \qquad
e_2 b_k = 
e_2 e_1 e_2 e_k = Z^2 e_1 e_2 e_k e_2 = 
b_k e_2,
$$
hence $Q(t_1,t_2)$ commutes with the generators of $H$,
hence with $H$.
Moreover,
 $H$ and $Q(t_1,t_2)$ together generate the group $Q(\ttt)$, and therefore
$H$ contains at least $2^{n-1}$ elements. 
We compute, for $3 \leq k < \ell $ (using $e_i^2 = t_i e_0$)
\begin{align*}
b_k^2 & = e_1 e_2 e_k e_1 e_2 e_k = Z^3 e_1^2 e_2^2 e_k^2 =
Z t_1 t_2 t_3 e_1 e_{2} e_k = Z t_1 t_2 t_3 \, b_k,
\\
b_k b_\ell & = e_1 e_2 e_k e_1 e_2 e_\ell =
Z e_1 e_2 e_\ell e_1 e_2 e_k = Z \, b_\ell b_k .
\end{align*}
These are the defining relations of the  
group $Q(Zt_1 t_2 t_3,\ldots, Z t_1 t_2 t_n)$,
hence, by Theorem \ref{prop:relations},
there is a surjective morphism from 
$Q(Z t_1 t_2 t_3,\ldots, Z t_1 t_2 t_n)$ to $H$,
which must be an isomorphism
for reasons of cardinality. 
 \end{proof}


In particular, the lemma implies
for $(t_1,t_2) = (1,1)$, or $(1,Z)$, or $(Z,Z)$:
\begin{align*}
Q(1,1,t_3,\ldots,t_n) & \cong Q_{2,0} \times_Z Q(Zt_3,\ldots,Zt_n) \cong
D_4  \times_Z Q(Zt_3,\ldots,Zt_n) ,
\\
Q(1,Z,t_3,\ldots,t_n) & \cong Q_{1,1} \times_Z Q(t_3,\ldots,t_n) \cong
D_4  \times_Z Q(t_3,\ldots,t_n) ,
\\
Q(Z,Z,t_3,\ldots,t_n) & \cong Q_{2,0} \times_Z Q(Zt_3,\ldots,Zt_n) \cong
Q  \times_Z Q(Z t_3,\ldots Z t_n). 
\end{align*}
From this, the first statements of the following corollary immediately
follow:

\begin{corollary}\label{cor:period}
For all $p,q \in \N$,
\begin{align*}
Q_{p+2,q} & \cong Q_{2,0} \times_Z Q_{q,p} \cong
D_4  \times_Z Q_{q,p}  ,
\\
Q_{p+1,q+1} & \cong Q_{1,1} \times_Z Q_{p,q} \cong
D_4  \times_Z Q_{p,q} ,
\\
Q_{p,q+2} & \cong Q_{0,2} \times_Z Q_{q,p} \cong
Q  \times_Z Q_{q,p} .
\end{align*}
Applying this several times, 
using notation $G^3_Z = G \times_Z G \times_Z G$, etc., 
\begin{align*}
Q_{p+4,q} & \cong Q_{2,0} \times_Z Q_{0,2}\times_Z Q_{p,q} \cong
D_4  \times_Z Q \times_Z Q_{p,q} \cong Q_{p,q+4},
\\
Q_{p+2,q+2} & \cong 
D_4 \times_Z D_4 \times_Z Q_{p,q} \cong
Q \times_Z Q \times_Z Q_{p,q} ,
\\
Q_{p+8,q} & \cong Q_{p,q+8} \cong Q_{p+4,q+4} \cong
(D_4)^4_Z \times_Z Q_{p,q} \cong Q^4_Z \times_Z Q_{p,q}.
\end{align*}
Moreover, for any $\Z/2 \Z$-graded group $G$,
$$
G \times_Z Q_{1,0} = G \times_{C_2} (C_2 \times C_2) \cong G \times C_2.
$$
\end{corollary}

\begin{proof}
The last statement is obvious from  $(C_2 \times C_2) / C_2 = C_2$.
\end{proof}

\begin{remark}\label{rk:isos}
In particular, the corollary gives the group isomorphisms
\begin{align*}
Q_{2,2} &  \cong D_4 \times_Z D_4 \cong Q \times_Z Q,
\\
Q_{1,2} & \cong C_4 \times_Z D \cong C_4 \times_Z Q.
\end{align*}
On the level of real Clifford algebras, this corresponds to 
the algebra isomorphisms
$\bH \otimes \bH  \cong M(2,\R)\otimes M(2,\R)$ (in turn, this algebra is
 isomorphic to $M(4,\R)$) and
$\bH \otimes \bC \cong M(2,\bC)$.
\end{remark}

\begin{notation}
To abbreviate, in the following,
we are going to write $D:=D_4$, $C:=C_4$,
$V:=C_2^2:=C_2 \times C_2$
(the Klein four-group),
 and since all direct products are
of the form $\times_Z$, we omit this sign, and just write it
as juxtaposition, and we write $D^2=DD$, etc. For instance,
$DQ = D \times_Z Q$,
$D^2 = D \times_Z D$. The isomorphisms from Remark \ref{rk:isos} thus
read $D^2 \cong Q^2$ and $DC_4 \cong QC_4$.
Note also that
$\times C_2 = \times_Z C_2^2$, so, e.g.,
$D C_2^2 = D \times_Z C_2^2 = D \times C_2$.
Using these isomorphisms and notation, every group $Q_{p,q}$ is
isomorphic to one in ``normal form''
$$
C_2, \quad D^k C_2^2, \quad D^k C_4, \quad
D^k Q, \quad Q C_2^2, \quad D^k Q C_2^2 .
$$
\end{notation}

\begin{theorem}\label{th:classification}
The discrete Clifford groups $Q_n$, resp.\ $Q_{0,n}$, are given by

\msk
\begin{center}
\begin{tabular}{l | l | l | l   |     l | l   |  l | l    |   l | l  | }
$n$ & $0$   &  $1$ &  $2$  & $3$ & $4$ & $5$ & $6$ & $7$ & $8$  \cr
\hline
$Q_n$  & $C_2$ & $C_2^2$ & $D$ & $DC_4$ & $DQ$ & $DQC_2^2$ &
$D^2 Q$ & $D^3 C_4$ & $D^4$ 
\\
$Q_{0,n}$  & $C_2$ & $C_4$ & $Q$ & $QC_2^2$ & $DQ$ & $D^2C_4$ &
$D^3$ & $D^3 C_2^2$ & $D^4$ 
\end{tabular}
\end{center}
\msk

\nin
and we have the periodicity relations:

\ssk
$Q_{n+4} \cong DQ \, Q_n, \quad $ 
$Q_{n+8} \cong D^4 Q_n, \quad $
$Q_{0,n+4} \cong DQ \, Q_{0,n}, \quad $
$Q_{0,n+8} \cong D^4 \, Q_{0,n}$.
\end{theorem}

\begin{proof}
Repeated application of Corollary \ref{cor:period}.
\end{proof}

For  sake of completeness, we give in  Section \ref{app:Qpq}
the full  ``periodic table'' of the groups $Q_{p,q}$, together with the
 corresponding
\href{https://en.wikipedia.org/wiki/Classification_of_Clifford_algebras#Classification}{classification of
Clifford algebras}.
Due to the 
relation $Q_{p+1,q+1} = Q_{p,q} \times_C D_4 = D \, Q_{p,q}$,
this table is best presented as a triangle (like Pascal's triangle).
This triangle
is not symmetric with respect to the ``main diagonal''
$Q_{n,n}$, but it is symmetric with respect to the ``secondary diagonal''
 $Q_{n+1,n}=C_2 \times (D_4)^n_Z = D^n \, C_2^2$, by the following lemma:

\begin{lemma} \label{la:diagonal} 
We have an isomorphism of (ungraded) groups:
$$
Q(t_1,\ldots,t_n) \cong Q(t_1,Zt_1 t_2,\ldots,Z t_1 t_n)
$$
In particular, there is a group isomorphism
$Q_{1+p,q} \cong Q_{1+q,p}$. 
\end{lemma}

\begin{proof}
The proof is similar to the one of Lemma \ref{la:Q-iso}:
define 
$$
b_1:=e_1, \quad b_2 := e_1 e_2 = e_{12}, \quad \ldots \quad
b_n:= e_{1n}.
$$
These elements
 satisfy the relations
 $b_1^2 = t_1$, and for $\ell > k \geq 2$,
$$
b_k^2 = e_1 e_k e_1 e_k = Z t_1 t_k, \quad
b_\ell b_k = e_1 e_\ell e_1 e_k = Z^3  e_1 e_k e_1 e_\ell =
Z b_k b_\ell ,
$$
so the group generated by $b_1,\ldots,b_n,Z$ is a quotient of
$Q(t_1,Zt_1 t_2,\ldots,Z t_1 t_n)$. On the other hand, 
since $b_1 b_2 = t_1 e_2$, etc., these elements generate $Q(\ttt)$,
so for reasons of cardinality,
$Q(\ttt) \cong Q(t_1,Zt_1 t_2,\ldots,Z t_1 t_n)$.
(Another proof is by starting from the isomorphism
$Q_2\cong Q_{1,1}$, and to note that it induces the claimed
symmetry.)
\end{proof}

\subsection{The even part}
Recall that the even part $G_0$ of a $\Z/2 \Z$-graded group
$G = G_0 \sqcup G_1$ is a (normal) subgroup.
If $G$ is finite and the grading non-trivial, then
$\vert G_0 \vert = \frac{1}{2} \vert G \vert$.
The element $Z$ belongs to $G_0$, but the grading of $G_0$ is
{\it a priori} trivial.

\begin{theorem}\label{th:even}
Let $n \geq 2$ and $\ttt \in \{ 1,Z \}^n$.
The even part of $Q(\ttt)$ is isomorphic to
$$
Q(\ttt)_0 \cong Q(Z t_1 t_2,\ldots,Z t_1 t_n).
$$
It follows that
$$
(Q_n)_0 \cong Q_{0,n-1} \cong (Q_{0,n})_0, \quad
(Q_{1+p,q})_0 \cong Q_{q,p}, \quad
(Q_{p,1+q})_0 \cong Q_{p,q}.
$$
Explicitly, the even part of $Q_n$, resp.\ $Q_{0,n}$, is given by

\begin{tabular}{l | l | l | l   |     l | l   |  l | l    |   l | l  | }
$n$ & $0$   &  $1$ &  $2$  & $3$ & $4$ & $5$ & $6$ & $7$ & $8$  \cr
\hline
$Q_n$  & $C_2$ & $C_2^2$ & $D$ & $DC_4$ & $DQ$ & $DQC_2^2$ &
$D^2 Q$ & $D^3 C_4$ & $D^4$ 
\\
$Q_{0,n}$  & $C_2$ & $C_4$ & $Q$ & $QC_2^2$ & $DQ$ & $D^2C_4$ &
$D^3$ & $D^3 C_2^2$ & $D^4$ 
\\
$(Q_n)_0 \cong (Q_{0,n})_0$ & $C_2$ & $C_2$ & $C_4$ & $Q$ & $Q C_2^2$ &
 $DQ$
& $D^2 C_4$
& $D^3$ & $D^3 C_2^2$
\end{tabular}
\end{theorem}

\begin{proof} 
For $k=2,\ldots,n$, define as above the even element
$b_k := e_1 e_k$. As noted above,
 $(b_k)^2 
= Z t_1 t_k$,
and, for $i \not= j$,
$$
b_i b_j = e_1 e_i e_1 e_j = Z^2  e_1^2 e_j e_i =  Z^3 e_1 e_j e_1 e_i = 
Z b_j b_i,
$$
so we get the desired relations for the generators
$b_2,\ldots,b_n$ of $Q(\ttt)_0$. For reasons
of cardinality
we conclude that $Q(\ttt)_0 \cong Q(Z t_1 t_2,\ldots,Z t_1 t_n)$.
Using Theorem \ref{th:classification}, we get the table for the even parts.
\end{proof}


\begin{remark}
As said above, {\it a priori}, the even part is ungraded; but since it is again a Clifford group,
it  ``remembers'' the grading of the preceding extensions, and
so does $((Q_n)_0)_0$, until we reach $C_0$, which is ungraded:
the table from the theorem gives us a chain of even parts
\begin{equation*}
\begin{matrix}
C_2 &  = & (C_4)_0 & = & ((Q)_0)_0 & = (((QC_2^2)_0)_0)_0 & =
& \ldots & = & (((((( D^4)_0)\ldots)_0 
\end{matrix}.
\end{equation*}
\end{remark}

\subsection{The periodic tables of discrete Clifford groups and of Clifford algebras}\label{app:Qpq}
Finally, here are the tables classifying discrete Clifford groups and Clifford algebras.
We present them in a form akin to Pascal's triangle: for the groups,
$$
\begin{matrix}
    & & Q_0 & & \\
    &Q_{1,0} & & Q_{0,1} & \\
 Q_{2,0} & & Q_{1,1} & & Q_{0,2} \\
 & & \ldots  & &
 \end{matrix}
 $$
 and for the algebras,
 $$
\begin{matrix}
    & & \K & & \\
    &\Cl_{1,0} (\K)& & \Cl_{0,1}(\K) & \\
 \Cl_{2,0}(\K) & & \Cl_{1,1} (\K)& & \Cl_{0,2}(\K) \\
 & & \ldots  & &
 \end{matrix}
  $$
For the following tables, cf.\ also \cite{S82, S, L}.

\begin{landscape} 

$ $

\setcounter{MaxMatrixCols}{20} 

\nin
{\bf 
Table of groups $Q_{p,q}$} 
($p+q$ runs vertically and $p-q$ runs horizontally,
notation
$V = C_2 \times C_2$, $C = C_4$, $D=D_4$, see above).
Note that the column labelled $1$ is a symmetry axis of both
tables (cf.\ Lemma \ref{la:diagonal} ).

$$
\begin{matrix}
 & 8 &7 & 6 & 5 & 4 & 3 & 2 & 1 & 0 & -1 & -2 & -3 & - 4 & - 5  & -6 & -7 & -8  
 & \mbox{Cardinality} \cr 
 0 & & & & & & &  &  & C_2 & &  & & & & & & &
 2 &
  \\
 1 & & & & & & & &  V  & & C  & & & &  & & &
 & 4
  \\
2 & & & & & &  & D & & D & & Q &  & &  & & &
& 8
\\
3 & & & & &  &{ DC  } & & { D V } & & 
{  D  C } & &
  Q V    & & & & &
  & 16
  \\
4 & & & &  &  { D Q }& & { D^2 } & & 
{ D^2 } & & 
 { DQ } & & D Q  &  & & &
 & 32
\\
5 & & & &
DQ V
& & D^2 C
& & D^2 V
& & D^2 C
& & D Q V 
& & D^2 C   & & &
& 64
\\
6 & & &
D^2  Q & & D^2 Q & & D^3 & & D^3 & & D^2 Q & & D^2 Q & & D^3  & &
& 128
\\
7 &  &
D^3 C  & & D^2Q  V  & & D^3 C  & & D^3 V  & & D^3 C 
& & D^2 Q  V  & & D^3 C  & & D^3  V &
& 256
\\
8 & D^4 & & D^3 Q & & D^3 Q & & D^4 & & D^4 & & D^3 Q & & D^3 Q
& & D^4 & & D^4 & 512
\\
\end{matrix}
$$

\medskip
\nin
Remark.
For very detailed information on the four groups of cardinality 16 listed in
line 3 of the table, see \cite{D}, page
529 for the group $DV$,  page 530 for  $DC$, and page 531 for $QV$.
(You can recognize them  in loc.cit.\  by looking at the structure of the center, containing
the derived group -- always of type $C_2$ -- and the quotient with respect to this
$C_2$ -- always of type $C_2^3$). 
E.g., $\Aut(DC)$ is exactly the hyperoctahedral group of order $48$, while 
$\Aut((QV)$ is bigger, of cardinality 192, and $\Aut(DV)$ is smaller, of cardinality
$32$.
For the groups
$D^2 = D\times_Z D$ and $DQ = D \times_Z Q$ (line 4), cf.\ loc.\ cit.,
p. 626, 627.

\msk

\nin
{\bf Table of classification of Clifford algebras $\Cl_{p,q}(\K)$}
(notation $\bC_\K= \K\oplus i \K$, 
$\bH_\K = \K\oplus i \K \oplus j \K \oplus k\K$, see above):
\ss

{\tiny
$$
\begin{matrix}
 & 8 &7 & 6 & 5 & 4 & 3 & 2 & 1 & 0 & -1 & -2 & -3 & - 4 & - 5  & -6 & -7 & -8 \cr 
 0 & & & & & & &  &  & \K & &  & & & & & &
  \\
 1 & & & & & & & & \K^2  & & \bC_\K & & & &  & & &
  \\
2 & & & & & &  & M_2(\K) & & M_2(\K) & & \bH_\K &  & &  & & &
\\
3 & & & & &  &{ M_2(\bC_\K) } & & { M_2(\K)^2 } & & 
{  M_2(\bC_\K) } & &
  \bH_\K^2   & & & & &
  \\
4 & & & &  &  { M_2(\bH_\K) }& & { M_4(\K) } & & 
{ M_4(\K) } & & 
 { M_2(\bH_\K) } & & M_2(\bH_\K)  &  & & &
\\
5 & & & &
M_2(\bH_\K)^2 
& & M_4(\bC_\K)
& & M_2(\K)^2 
& & M_4(\bC_\K)
& & M_2(\bH)^2
& & M_4(\bC_\K)  & & &
\\
6 & & &
M_4(\bH_\K) & & M_4(\bH_\K) & & M_{8}(\K) & & M_{8}(\K) & & 
M_4(\bH_\K) & & M_4(\bH_\K) & & M_8(\K)  & &
\\
7 &  &
M_8(\bC_\K) & & M_4(\bH)^2 & & M_8(\bC_\K) & & 
M_8(\K)^2 & & M_8(\bC_K)
& & M_4(\bH_\K)^2 & & M_8(\bC_\K) & & M_8(\K)^2 &
\\
8 & M_{16}(\K) & & M_8(\bH_\K) & & M_8(\bH_\K) & & M_{16}(\K) & & 
M_{16}(\K) & & M_8(\bH_\K) & & M_8(\bH_\K)
& & M_{16}(\K) & & M_{16}(\K)
\\
\end{matrix}
$$
}
\end{landscape}

\end{document}